\documentclass[a4paper]{amsart}

\usepackage{mathtools}

\def\mylabelonoff{off}
\def\allowdisbrkyesno{yes}
\def\numberingtheoremsectionyesno{yes}
\def\numberingequationsectionyesno{yes}
\def\pagesizeextendednormal{extended}

\def\reportudemathyesno{no}
\def\reportudemathnumber{SM-UDE-811}
\def\reportudemathyear{2017}
\def\reportudematheingang{\mydate}

\def\mytitle{On the Maxwell and Friedrichs/Poincar\'e Constants in ND}
\def\mytitlerepude{On the Maxwell and Friedrichs/Poincar\'e Constants in ND}
\def\myshorttitle{On the Maxwell Constants in $\rN$}
\def\myauthorone{Dirk Pauly}
\def\myauthors{\myauthorone}
\def\myaddressone{Fakult\"at f\"ur Mathematik,
Universit\"at Duisburg-Essen, Campus Essen, Germany}
\def\myemailone{dirk.pauly@uni-due.de}
\def\mykeywords{Maxwell's equations, Maxwell constant,
second Maxwell eigenvalue, electro statics, magneto statics,
Poincar\'e inequality, Friedrichs inequality, 
Poincar\'e constant, Friedrichs constant}
\def\mysubjclass{35A23, 35Q61, 35E10, 35F15, 35R45, 46E40, 53A45}
\def\mydate{\today}

\usepackage[mathscr]{eucal}
\usepackage[english]{babel}
\usepackage{a4,exscale,ifthen,amsfonts,amssymb,amsmath,amscd,graphicx,color}
\usepackage{nicefrac,tikz,fancyhdr,caption,array,multirow,multicol,booktabs,algorithm,algorithmic}
\usepackage[all]{xy}

\ifthenelse{\equal{\mylabelonoff}{on}}
{\newcommand{\mylabel}[1]{\label{#1}\fbox{{\sf #1}}}}
{\newcommand{\mylabel}[1]{\label{#1}}}
\ifthenelse{\equal{\allowdisbrkyesno}{yes}}
{\allowdisplaybreaks}
{}

\ifthenelse{\equal{\pagesizeextendednormal}{extended}}
{\setlength{\textwidth}{16cm}
\setlength{\textheight}{22cm}
\setlength{\oddsidemargin}{-0.2cm}
\setlength{\evensidemargin}{-0.2cm}}
{}
\ifthenelse{\equal{\numberingequationsectionyesno}{yes}}
{\numberwithin{equation}{section}}
{}

\newcommand{\ovl}[1]{\overline{#1}}
\newcommand{\unl}[1]{\underline{#1}}

\DeclareMathOperator{\diam}{diam}

\newcommand{\cp}{c_{\mathsf{p}}}

\newcommand{\cf}{c_{\mathsf{f}}}

\newcommand{\setb}[2]{\big\{#1\,:\,#2\big\}}

\ifthenelse{\equal{\numberingtheoremsectionyesno}{yes}}
{\newtheorem{lem}{Lemma}[section]}
{\newtheorem{lem}{Lemma}}

\newtheorem{theo}[lem]{Theorem}
\newtheorem{cor}[lem]{Corollary}
\newtheorem{rem}[lem]{Remark}

\newenvironment{acknow}{{\vspace*{1cm}\noindent\sc Acknowledgements }}{}

\newcommand{\om}{\Omega}

\newcommand{\ga}{\Gamma}

\newcommand{\eps}{\epsilon}

\newcommand{\reals}{\mathbb{R}}

\newcommand{\rt}{\reals^{3}}
\newcommand{\rN}{\reals^{N}}

\newcommand{\rNtN}{\reals^{N\times N}}

\newcommand{\rttt}{\reals^{3\times3}}

\newcommand{\oh}{\nicefrac{1}{2}}
\newcommand{\moh}{-\oh}

\newcommand{\ot}{\leftarrow}

\DeclareMathOperator{\id}{id}

\DeclareMathOperator{\skw}{skw}

\DeclareMathOperator{\supp}{supp}
\DeclareMathOperator{\dist}{dist}
\DeclareMathOperator{\A}{A}

\DeclareMathOperator{\As}{\A^{*}}

\DeclareMathOperator{\cA}{\mathcal{A}}

\DeclareMathOperator{\cAs}{\cA^{*}}

\DeclareMathOperator{\p}{\partial}
\DeclareMathOperator{\na}{\nabla}

\DeclareMathOperator{\rot}{rot}

\DeclareMathOperator{\curl}{curl}

\DeclareMathOperator{\divergence}{div}
\renewcommand{\div}{\divergence}

\DeclareMathOperator{\ed}{d}
\DeclareMathOperator{\cd}{\delta}

\newcommand{\csymbol}{\mathsf{C}}
\newcommand{\cgen}[3]{\overset{#1}{\csymbol}{}^{#2}_{#3}}

\newcommand{\ci}{\cgen{}{\infty}{}}
\newcommand{\cic}{\cgen{\circ}{\infty}{}}

\newcommand{\ct}{\cgen{}{2}{}}

\newcommand{\cicom}{\cic(\om)}

\newcommand{\lsymbol}{\mathsf{L}}
\newcommand{\lgen}[3]{\overset{#1}{\lsymbol}{}^{#2}_{#3}}

\newcommand{\lt}{\lgen{}{2}{}}

\newcommand{\li}{\lgen{}{\infty}{}}

\newcommand{\lteps}{\lgen{}{2}{\eps}}

\newcommand{\ltom}{\lt(\om)}

\newcommand{\liom}{\li(\om)}

\newcommand{\ltepsom}{\lteps(\om)}

\newcommand{\hsymbol}{\mathsf{H}}

\newcommand{\hgen}[3]{\overset{#1}{\hsymbol}{}^{#2}_{#3}}

\newcommand{\ho}{\hgen{}{1}{}}

\newcommand{\hoc}{\hgen{\circ}{1}{}}

\newcommand{\hoom}{\ho(\om)}

\newcommand{\Hoom}{\Ho(\om)}
\newcommand{\hocom}{\hoc(\om)}

\newcommand{\dsymbol}{\mathsf{D}}

\newcommand{\harmsymbol}{\mathcal{H}}
\newcommand{\harmgen}[3]{\overset{#1}{\harmsymbol}{}^{#2}_{#3}}

\newcommand{\harmd}{\harmgen{}{}{\mathsf{D}}}

\newcommand{\harmdom}{\harmd(\om)}

\newcommand{\norm}[1]{|#1|}

\newcommand{\bnorm}[1]{\big|#1\big|}

\newcommand{\normltom}[1]{\norm{#1}_{\ltom}}

\newcommand{\normltepsom}[1]{\norm{#1}_{\ltepsom}}

\newcommand{\scp}[2]{\langle#1,#2\rangle}

\newcommand{\scpltom}[2]{\scp{#1}{#2}_{\ltom}}

\newcommand{\preprintudemath}[5]{
\thispagestyle{empty}
\Large
\begin{center}SCHRIFTENREIHE DER FAKULT\"AT F\"UR MATHEMATIK\end{center}
\vspace*{5mm}
\begin{center}#1\end{center}
\vspace*{5mm}
\begin{center}by\end{center}
\begin{center}#2\end{center}
\vspace*{5mm}
\begin{center}#3\hspace{80mm}#4\end{center}
\newpage
\thispagestyle{empty}
\vspace*{210mm}
Received: #5
\newpage
\addtocounter{page}{-2}
\normalsize}

\title[\sc\myshorttitle]{\Large\sf\mytitle}
\author{\myauthorone}
\address{\myaddressone}
\email[\myauthorone]{\myemailone}
\subjclass{\mysubjclass}
\date{\mydate}

\renewcommand{\cic}{\mathring\csymbol^{\infty}}
\renewcommand{\hoc}{\mathring\hsymbol^1}

\DeclareMathOperator{\Az}{A_{0}}
\DeclareMathOperator{\Azs}{A_{0}^{*}}
\DeclareMathOperator{\Ao}{A_{1}}
\DeclareMathOperator{\Aos}{A_{1}^{*}}

\newcommand{\wto}{\xrightharpoonup}
\newcommand{\cptemb}{\hookrightarrow\hspace*{-0.85em}\hookrightarrow}
\newcommand{\harmdepsom}{\harmsymbol_{\mathsf{D},\eps}(\om)}
\newcommand{\harmdidom}{\harmsymbol_{\mathsf{D},\id}(\om)}
\newcommand{\harmnepsom}{\harmsymbol_{\mathsf{N},\eps}(\om)}
\newcommand{\Harmdepsom}[1]{\harmsymbol^{#1}_{\mathsf{D},\eps}(\om)}
\newcommand{\Harmnepsom}[1]{\harmsymbol^{#1}_{\mathsf{N},\eps}(\om)}
\newcommand{\scpltepsom}[2]{\scp{#1}{#2}_{\ltepsom}}

\newcommand{\epsu}{\unl{\eps}}
\newcommand{\epso}{\ovl{\eps}}
\newcommand{\epsh}{\hat{\eps}}
\newcommand{\cpc}{\cf}
\newcommand{\cmeps}{c_{\mathsf{m},\eps}}
\newcommand{\cmteps}{c_{\mathsf{m,t},\eps}}

\newcommand{\cmneps}{c_{\mathsf{m,n},\eps}}

\newcommand{\Lt}[1]{\lsymbol^{2,#1}}
\newcommand{\Lteps}[1]{\lsymbol^{2,#1}_{\eps}}
\newcommand{\Ltmu}[1]{\lsymbol^{2,#1}_{\mu}}
\newcommand{\genHo}[1]{\hsymbol^{1,#1}}
\newcommand{\Cic}[1]{\mathring\csymbol^{\infty,#1}}
\newcommand{\Hoc}[1]{\mathring\hsymbol^{1,#1}}
\newcommand{\eD}[1]{\dsymbol^{#1}}
\newcommand{\eDz}[1]{\dsymbol^{#1}_{0}}
\newcommand{\cD}[1]{\Delta^{#1}}
\newcommand{\cDz}[1]{\Delta^{#1}_{0}}
\newcommand{\eDc}[1]{\mathring\dsymbol^{#1}}
\newcommand{\eDcz}[1]{\mathring\dsymbol^{#1}_{0}}
\newcommand{\cDc}[1]{\mathring\Delta^{#1}}
\newcommand{\cDcz}[1]{\mathring\Delta^{#1}_{0}}
\newcommand{\Ltom}[1]{\Lt{#1}(\om)}
\newcommand{\Ltepsom}[1]{\Lteps{#1}(\om)}
\newcommand{\Ltmuom}[1]{\Ltmu{#1}(\om)}
\renewcommand{\Hoom}[1]{\genHo{#1}(\om)}
\newcommand{\Cicom}[1]{\Cic{#1}(\om)}
\newcommand{\Hocom}[1]{\Hoc{#1}(\om)}
\newcommand{\eDom}[1]{\eD{#1}(\om)}
\newcommand{\eDzom}[1]{\eDz{#1}(\om)}
\newcommand{\cDom}[1]{\cD{#1}(\om)}
\newcommand{\cDzom}[1]{\cDz{#1}(\om)}
\newcommand{\eDcom}[1]{\eDc{#1}(\om)}
\newcommand{\eDczom}[1]{\eDcz{#1}(\om)}
\newcommand{\cDcom}[1]{\cDc{#1}(\om)}
\newcommand{\cDczom}[1]{\cDcz{#1}(\om)}
\newcommand{\Cteps}[1]{c_{\mathsf{t},#1,\eps}}
\newcommand{\Cneps}[1]{c_{\mathsf{n},#1,\eps}}
\newcommand{\Ctmu}[1]{c_{\mathsf{t},#1,\mu}}

\newcommand{\Ct}[1]{c_{\mathsf{t},#1}}
\newcommand{\Cn}[1]{c_{\mathsf{n},#1}}
\newcommand{\Cedcteps}[1]{c_{\edc,\mathsf{t},#1,\eps}}
\newcommand{\Cedctmu}[1]{c_{\edc,\mathsf{t},#1,\mu}}
\newcommand{\Cedct}[1]{c_{\edc,\mathsf{t},#1}}
\newcommand{\tCedcteps}[1]{\tilde{c}_{\edc,\mathsf{t},#1,\eps}}
\newcommand{\tCedctmu}[1]{\tilde{c}_{\edc,\mathsf{t},#1,\mu}}
\newcommand{\tCedct}[1]{\tilde{c}_{\edc,\mathsf{t},#1}}
\renewcommand{\H}{\mathsf{H}}
\renewcommand{\A}{\mathrm{A}}
\renewcommand{\As}{\A^{*}}

\renewcommand{\cA}{\mathcal{A}}
\renewcommand{\cAs}{\cA^{*}}

\renewcommand{\Az}{\mathrm{A}_{0}}
\renewcommand{\Azs}{\Az^{*}}

\DeclareMathOperator{\cAz}{\mathcal{A}_{0}}
\DeclareMathOperator{\cAzs}{\mathcal{A}_{0}^{*}}

\renewcommand{\Ao}{\mathrm{A}_{1}}
\renewcommand{\Aos}{\Ao^{*}}

\DeclareMathOperator{\cAo}{\mathcal{A}_{1}}
\DeclareMathOperator{\cAos}{\mathcal{A}_{1}^{*}}

\DeclareMathOperator{\edc}{\mathring\ed}
\DeclareMathOperator{\cdc}{\mathring\cd}
\newcommand{\longvec}{\overrightarrow}
\DeclareMathOperator{\adj}{adj}
\DeclareMathOperator{\adjt}{\adj^{\top}}

\begin{document}


\ifthenelse{\equal{\reportudemathyesno}{yes}}
{\preprintudemath{\mytitlerepude}{\myauthors}{\reportudemathnumber}{\reportudemathyear}{\reportudematheingang}}
{}


\begin{abstract}
We prove that for bounded and convex domains in arbitrary dimensions, 
the Maxwell constants are bounded from below and above by Friedrichs' and Poincar\'e's constants, respectively.
Especially, the second positive Maxwell eigenvalues in ND
are bounded from below by the square root of the second Neumann-Laplace eigenvalue.
\end{abstract}

\vspace*{-10mm}
\maketitle
\tableofcontents


\section{Introduction}

\subsection{Maxwell's Equations}

Maxwell's equations are fundamental in physics 
and play an important role for mathematical physics itself.
In a domain $\om\subset\rt$ (open and connected set)
with boundary $\ga$ these famous equations read
for the pair of the electric and magnetic field $(E,H)$
\begin{align*}
\curl E+\p_{t}\mu H&=G,
&
-\curl H+\p_{t}\eps E&=F
&
\text{in }&(0,T)\times\om,\\
\div\eps E&=f,
&
\div\mu H&=g
&
\text{in }&(0,T)\times\om,\\
n\times E&=0,
&
n\cdot\mu H&=0
&
\text{at }&(0,T)\times\ga,
\end{align*}
where we have already eliminated the fields $D$ and $B$
by the constitutive laws $D=\eps E$ and $B=\mu H$, respectively.
Moreover, physically meaningful is $F=-j$ as current density and $G=0$
as well as $f=\rho$ as charge density and $g=0$. 
Furthermore, initial conditions have to be imposed on $E(0)$ and $H(0)$ in $\om$.
Note that in the non-stationary case the two 
divergence equations are redundant by the two $\curl$-equations
and the complex property $\div\rot=0$.
Moreover, the second normal boundary condition for $\mu H$ is 
already given by the first tangential boundary condition for $E$ and the first $\curl$-equation 
as $n\times E=0$ implies $n\cdot\curl E=0$ at $(0,T)\times\ga$.
In the time-harmonic setting 
(all fields depend on a fixed frequency $\omega>0$ in a sinusodial way)
we have
\begin{align*}
\curl E+i\omega\mu H&=G,
&
-\curl H+i\omega\eps E&=F
&
\text{in }&\om,\\
\div\eps E&=f,
&
\div\mu H&=g
&
\text{in }&\om,\\
n\times E&=0,
&
n\cdot\mu H&=0
&
\text{at }&\ga,
\end{align*}
where the divergence equations and the second boundary condition are still redundant.
Finally, the electro-magnto static equations are given by 
\begin{align*}
\curl E&=G,
&
\curl H&=-F
&
\text{in }&\om,\\
\div\eps E&=f,
&
\div\mu H&=g
&
\text{in }&\om,\\
n\times E&=0,
&
n\cdot\mu H&=0
&
\text{at }&\ga
\end{align*}
and we emphasize that here the divergence equations and the boundary condition for $H$ 
are no longer redundant as the systems completely decouples into two separate systems,
the electro static equations for the electric field $E$ 
and the magneto static equations for magnetic field $H$.

Proper solution theories in the sense of Hadamard, i.e.,
unique and continuous solvability,
are well known, see e.g. \cite{leisbook}.
In the static and time-harmonic situations the essential tool is the
so-called Maxwell estimate (setting $H=E$ and $\mu=\eps$)
$$\normltepsom{E}^2
\leq\cmeps\big(\normltom{\curl E}^2+\normltom{\div\eps E}^2\big),$$
see \eqref{intromaxestelec} and \eqref{intromaxestmag},
being valid for all $E\in\ltom$ with $\curl E\in\ltom$ and $\div\eps E\in\ltom$
as well as either $n\times E|_{\ga}=0$ or $n\cdot\eps E|_{\ga}=0$
such that $E$ is perpendicular to the possible kernels
$\harmdepsom$ or $\harmnepsom$, respectively,
the so-called Dirichlet or Neumann fields.
Of course, all terms have to be understood in a weak way
which we define below in a suitable Sobolev setting.
Obviously, the best constant $\cmeps$ resp. $(1+\cmeps^2)^{\nicefrac{1}{2}}$ 
is the norm of the respective bounded inverse, mapping
the right hand sides to the solution $E$ (resp. $H$).

A more general situation can be considered if we assume
$\om$ to be a Riemannan manifold of dimension $N$.
In particular $\om$ may be an open subset of $\rN$
or some $N$-dimensional surface in $\reals^{M}$.
Then Maxwell's equations can be expressed independently 
of special coordinates by the calculus of differential forms
using the exterior derivative $\ed$ and co-derivative $\cd=\pm*\cd*$
as well as the Hodge star operator $*$.
Focusing on the static equations we have
for a $q$-from $\xi$ and a $(q+1)$-form $\zeta$
\begin{align*}
\ed\zeta&=\phi,
&
\text{in }&\om,\\
\cd\eps\,\zeta&=\theta,
&
\text{in }&\om,\\
\iota^{*}\zeta&=0,
&
\text{on }&\ga,
\end{align*}
where $\iota$ is the canonical embedding of the boundary manifold $\ga$ into $\ovl{\om}$
and $\iota^{*}$ its pull-back. 
For $N=3$, $q=1$ and the vector proxy $E=\vec\zeta$
we get back the classical electro static formulation of vector analysis from above.
For $N=3$, $q=2$ and the vector proxy $\mu H=\vec\zeta$ (setting $\eps=\mu^{-1}$)
we get back the classical magneto static formulation.
Here, the crucial tool for a proper solution theory is the so-called generalized Maxwell estimate
$$\normltepsom{\zeta}^2
\leq\cmeps\big(\normltom{\ed\zeta}^2+\normltom{\cd\eps\,\zeta}^2\big),$$
see \eqref{intromaxestelecdf},
being valid for all $\zeta\in\ltom$ with $\ed\zeta\in\ltom$ and $\div\eps\,\zeta\in\ltom$
such that the related boundary and kernel conditions hold in a suitable weak Sobolev sense.

\subsection{The Maxwell Constants}

Let $\om\subset\rt$ be a bounded weak Lipschitz domain,
see \cite[Defintion 2.3]{bauerpaulyschomburgmcpweaklip} for an exact definition.
We denote the standard Lebesgue and Sobolev spaces by
$\ltom$, $\hoom$, which might be scalar-, vector-, or tensor-valued, 
and by  $\hsymbol(\curl,\om)$, $\hsymbol(\div,\om)$
the respective Sobolebv spaces for the rotation $\curl$
and the divergence $\div$.
Moreover, we introduce homogeneous scalar, tangential, and normal boundary conditions in the spaces
$\hocom$, $\mathring\hsymbol(\curl,\om)$, $\mathring\hsymbol(\div,\om)$, respectively,
which are defined as closures of $\cicom$-functions, vector, or tensor fields
under the corresponding graph norms.
Moreover, let $\eps:\om\to\rttt$ be a symmetric, $\liom$-bounded, 
and uniformly positive definite matrix field.

It is well known that the tangential version of Weck's selection theorem, stating that the embedding
\begin{align}
\mylabel{introallcompactembmaxt}
\mathring\hsymbol(\curl,\om)\cap\eps^{-1}\hsymbol(\div,\om)\cptemb\ltom
\end{align}
is compact, see \cite{weckmax,picardcomimb,webercompmax,witschremmax,picardweckwitschxmas,bauerpaulyschomburgmcpweaklip}, 
is the crucial tool of any analysis for static or time-harmonic Maxwell equations.
Especially, \eqref{introallcompactembmaxt} implies by a standard indirect argument the following important 
Maxwell estimate for tangential boundary conditions:
There exists a constant $\cmteps>0$ such that for all
$E\in\mathring\hsymbol(\curl,\om)\cap\eps^{-1}\hsymbol(\div,\om)\cap\harmdepsom^{\bot_{\ltepsom}}$
\begin{align}
\mylabel{intromaxestelec}
\normltepsom{E}
\leq\cmteps\big(\normltom{\curl E}^2+\normltom{\div\eps E}^2\big)^{\nicefrac{1}{2}}
\end{align}
holds, where the kernel space of (harmonic) Dirichlet fields is denoted by
$$\harmdepsom:=\setb{E\in\mathring\hsymbol(\curl,\om)\cap\eps^{-1}\hsymbol(\div,\om)}{\curl E=0,\,\div\eps E=0}.$$
Note that $\harmdepsom$ is finite dimensional by \eqref{introallcompactembmaxt} as its unit ball is compact.
We also introduce the weighted $\eps$-$\ltom$-scalar product 
$\scpltepsom{\,\cdot\,}{\,\cdot\,}:=\scpltom{\eps\,\cdot\,}{\,\cdot\,}$
and the corresponding induced weighted 
$\eps$-$\ltom$-norm $\normltepsom{\,\cdot\,}:=\scpltepsom{\,\cdot\,}{\,\cdot\,}^{\nicefrac{1}{2}}=\normltom{\eps^{\nicefrac{1}{2}}\,\cdot\,}$. 
If we equip $\ltom$ with this weighted scalar product we write $\ltepsom$.
Moreover, $\bot_{\ltepsom}$ denotes orthogonality with respect to the $\eps$-$\ltom$-scalar product.
If $\eps$ equals the identity $\id$, it will be skipped in our notations,
e.g., we write $\bot_{\ltom}$ and $\harmdom=\harmdidom$.

The fact that a compact embedding implies by an indirect argument 
a corresponding Friedrichs/Poincar\'e type estimate, is a well known and powerful concept. 
Prominent examples are the Friedrichs and Poincar\'e estimates itself, i.e.,
\begin{align}
\mylabel{intropoincarehoc}
\exists\,\cpc&>0&\forall\,u&\in\hocom&\normltom{u}&\leq\cpc\normltom{\na u},\\
\mylabel{intropoincareho}
\exists\,\cp&>0&\forall\,v&\in\hoom\cap\reals^{\bot_{\ltom}}&\normltom{v}&\leq\cp\normltom{\na v},
\end{align}
which follow immediately using Rellich's selection theorem, i.e., the compactness of
\begin{align}
\mylabel{introallcompactembrellich}
\hocom\subset\hoom\cptemb\ltom.
\end{align}
For the best possible constants it holds
$$\cpc^2=\frac{1}{\lambda_{1}}<\frac{1}{\mu_{2}}=\cp^2,$$
where 
$$\lambda_{1}=\min_{u\in\hocom}\frac{\normltom{\na u}^2}{\normltom{u}^2},\qquad
\mu_{2}=\min_{v\in\hoom\cap\reals^{\bot_{\ltom}}}\frac{\normltom{\na v}^2}{\normltom{v}^2}$$
is the first Dirichlet resp. second Neumann eigenvalue of the negative Laplacian,
see, e.g., \cite{filonovdirneulapeigen} and the literature cited there.
Analogously to \eqref{introallcompactembmaxt} and \eqref{intromaxestelec},
the normal version of Weck's selection theorem,  i.e., the compactness of the embedding
\begin{align}
\mylabel{introallcompactembmaxn}
\hsymbol(\curl,\om)\cap\eps^{-1}\mathring\hsymbol(\div,\om)\cptemb\ltom,
\end{align}
shows the corresponding Maxwell estimate for normal boundary conditions:
There exists a constant $\cmneps>0$ such that for all
$H\in\hsymbol(\curl,\om)\cap\eps^{-1}\mathring\hsymbol(\div,\om)\cap\harmnepsom^{\bot_{\ltepsom}}$
\begin{align}
\mylabel{intromaxestmag}
\normltepsom{H}
\leq\cmneps\big(\normltom{\curl H}^2+\normltom{\div\eps H}^2\big)^{\nicefrac{1}{2}},
\end{align}
where we define the finite dimensional kernel space of (harmonic) Neumann fields by
$$\harmnepsom:=\setb{H\in\hsymbol(\curl,\om)\cap\eps^{-1}\mathring\hsymbol(\div,\om)}{\curl H=0,\,\div\eps H=0}.$$
Similarly to the Friedrichs and Poincare constants we always assume the best constants, i.e.,
\begin{align*}
\frac{1}{\cmteps^2}
&=\min_{E}\frac{\normltom{\curl E}^2+\normltom{\div\eps E}^2}{\normltepsom{E}^2},
&
\frac{1}{\cmneps^2}
&=\min_{H}\frac{\normltom{\curl H}^2+\normltom{\div\eps H}^2}{\normltepsom{H}^2},
\end{align*}
where the first minimum is taken over
$E\in\mathring\hsymbol(\curl,\om)\cap\eps^{-1}\hsymbol(\div,\om)\cap\harmdepsom^{\bot_{\ltepsom}}$
and the second over 
$H\in\hsymbol(\curl,\om)\cap\eps^{-1}\mathring\hsymbol(\div,\om)\cap\harmnepsom^{\bot_{\ltepsom}}$.

In \cite{paulymaxconst0,paulymaxconst1,paulymaxconst2} we have shown that for convex $\om$
and, provided that always the best possible constants are chosen, the estimates
\begin{align}
\mylabel{introconstest}
\frac{\cpc}{\epsh^3}\leq\cmteps,\cmneps\leq\cp\epsh\leq\frac{\diam(\om)}{\pi}\epsh
\end{align}
hold, where 
\begin{align}
\mylabel{epsuotwo}
\epsh:=\max\{\epsu,\epso\},
\end{align}
and the lower and upper bounds $\epsu,\epso>0$ for $\eps$ are defined by
\begin{align}
\mylabel{epsuoone}
\forall\,E\in\ltom\qquad
\epsu^{-2}\normltom{E}^2\leq\scpltom{\eps E}{E}\leq\epso^2\normltom{E}^2,
\end{align}
which exist by our assumptions.
Note that convex domains are even strong Lipschitz, see, e.g.,
\cite[Corollary 1.2.2.3]{grisvardbook}
and topologically trivial, i.e., they satisfy $\harmdepsom=\harmnepsom=\{0\}$
as $\dim\harmnepsom$ resp. $\dim\harmdepsom$
is given by the first resp. second Betti number of $\om$.

The aim of the paper at hand is to generalize and improve the estimates \eqref{introconstest}
for the Maxwell constants to convex domains $\om\subset\rN$.
In $\rN$ it is useful to work within the setting of alternating differential forms
of general order $q\in\{0,\dots,N\}$. More precisely, let
$\om\subset\rN$ be a bounded weak Lipschitz domain, 
whose definition is easily modified from the 3D case,
see again \cite[Defintion 2.3]{bauerpaulyschomburgmcpweaklip}.
We denote the standard Lebesgue and Sobolev spaces by $\Ltom{q}$, and
\begin{align*}
\eDom{q}:=\hsymbol^{q}(\ed,\om)&=\setb{\omega\in\Ltom{q}}{\ed\omega\in\Ltom{q+1}},\\
\cDom{q}:=\hsymbol^{q}(\cd,\om)&=\setb{\omega\in\Ltom{q}}{\cd\omega\in\Ltom{q-1}},
\end{align*}
where $\ed$ is the exterior derivative, $\delta:=(-1)^{(q-1)N}*\ed*$ the co-derivative,
and $*$ the Hodge-star-operator.
Moreover, we introduce so-called homogeneous tangential and normal boundary conditions in the spaces
\begin{align*}
\eDcom{q}=\mathring\hsymbol^{q}(\ed,\om),\quad
\cDcom{q}=\mathring\hsymbol^{q}(\cd,\om),
\end{align*}
respectively, which are defined as before as closures of $\Cicom{q}$-forms under the corresponding graph norms.
A vanishing derivative will always be indicated by an index zero at the lower right corner, e.g.,
$$\eDzom{q}:=\setb{\omega\in\eDom{q}}{\ed\omega=0},\quad
\cDczom{q}:=\setb{\omega\in\cDcom{q}}{\cd\omega=0}.$$
It holds
\begin{align}
\mylabel{starDdef}
*\,\eDom{q}=\cDom{N-q},\quad
*\,\cDom{q}=\eDom{N-q},\quad
*\,\eDcom{q}=\cDcom{N-q},\quad
*\,\cDcom{q}=\eDcom{N-q}.
\end{align}
Inner products and hence norms are defined by
\begin{align*}
\scp{\omega}{\zeta}_{\Ltom{q}}
&:=\int_{\om}\omega\wedge*\,\bar{\zeta},&
\omega,\zeta
&\in\Ltom{q},\\
\scp{\omega}{\zeta}_{\eDom{q}}
&:=\scp{\omega}{\zeta}_{\Ltom{q}}+\scp{\ed\omega}{\ed\zeta}_{\Ltom{q+1}},&
\omega,\zeta
&\in\eDom{q},\\
\scp{\omega}{\zeta}_{\cDom{q}}
&:=\scp{\omega}{\zeta}_{\Ltom{q}}+\scp{\cd\omega}{\cd\zeta}_{\Ltom{q-1}},&
\omega,\zeta
&\in\cDom{q}.
\end{align*}
We emphasize that for $q$-forms $\omega$ given in Cartesian coordinates (identity map/chart), i.e.,
$$\omega=\sum_{I}\omega_{I}\ed x^{I}$$
with ordered multi-indices $I=(i_{1},\dots,i_{q})$,
we have $\omega\in\Ltom{q}$ if and only if $\omega_{I}\in\ltom$ for all $I$.
The inner product for $\omega,\zeta\in\Ltom{q}$ is given by
$$\scp{\omega}{\zeta}_{\Ltom{q}}
=\int_{\om}\omega\wedge*\,\bar{\zeta}
=\sum_{I}\int_{\om}\omega_{I}\bar{\zeta}_{I}
=\sum_{I}\scpltom{\omega_{I}}{\zeta_{I}}
=\scpltom{\vec\omega}{\vec\zeta},$$
where we introduce the vector proxy notation
$$\vec\omega=[\omega_{I}]_{I}\in\lt(\om;\reals^{N_{q}}),\qquad
N_{q}:=\binom{N}{q}.$$

The spaces $\Ltepsom{q}$ with the inner products 
$\scp{\,\cdot\,}{\,\cdot\,}_{\Ltepsom{q}}=\scp{\eps\,\cdot\,}{\,\cdot\,}_{\Ltom{q}}$
are defined in the same way as for vector or tensor fields, where
$\eps:\Ltom{q}\to\Ltom{q}$ is a symmetric, bounded, and uniformly positive definite 
transformation on $q$-forms. 
Such transformations will be called admissible.
All other definitions and notations concerning $\eps$
carry over to $q$-forms as well, e.g., we have \eqref{epsuoone} and \eqref{epsuotwo}.
More precisely, by the assumptions on $\eps$ we have
\begin{align}
\mylabel{epsuoonedf}
\exists\,\epsu,\epso>0\quad\forall\,\omega\in\Ltom{q}\qquad
\epsu^{-2}\norm{\omega}_{\Ltom{q}}^2\leq\scp{\eps\,\omega}{\omega}_{\Ltom{q}}\leq\epso^2\norm{\omega}_{\Ltom{q}}^2
\end{align}
and we note $\norm{\omega}_{\Ltepsom{q}}^2=\scp{\eps\,\omega}{\omega}_{\Ltom{q}}=\norm{\eps^{\nicefrac{1}{2}}\omega}_{\Ltom{q}}^2$
as well as $\norm{\eps\,\omega}_{\Ltom{q}}=\norm{\eps^{\nicefrac{1}{2}}\omega}_{\Ltepsom{q}}$. Thus, for all $\omega\in\Ltom{q}$
\begin{align}
\mylabel{epsuotwodf}
\epsu^{-1}\norm{\omega}_{\Ltom{q}}\leq\norm{\omega}_{\Ltepsom{q}}\leq\epso\norm{\omega}_{\Ltom{q}},\quad
\epsu^{-1}\norm{\omega}_{\Ltepsom{q}}\leq\norm{\eps\,\omega}_{\Ltom{q}}\leq\epso\norm{\omega}_{\Ltepsom{q}}.
\end{align}

As in the vector-valued case we can also define the Sobolev spaces $\Hoom{q}$ resp. $\Hocom{q}$
component-wise by defining $\omega\in\Hoom{q}$ resp. $\omega\in\Hocom{q}$ 
if and only if $\omega_{I}\in\hoom$ resp. $\omega_{I}\in\hocom$ for all $I$. 
In these cases we have for $n=1,\dots,N$
$$\p_{n}\omega=\sum_{I}\p_{n}\omega_{I}\ed x^{I}$$
and we utilize the vector proxy notation also for the gradient, i.e.,
$$\na\vec\omega=[\p_{n}\omega_{I}]_{n,I}
=[\dots\na\omega_{I}\dots]_{I}\in\lt(\om;\reals^{N\times N_{q}}).$$
Hence, for $\omega,\zeta\in\Hoom{q}$
\begin{align*}
\scp{\omega}{\zeta}_{\Hoom{q}}
&:=\scp{\omega}{\zeta}_{\Ltom{q}}
+\sum_{n=1}^{N}\scp{\p_{n}\omega}{\p_{n}\zeta}_{\Ltom{q}}
=\int_{\om}\omega\wedge*\,\bar{\zeta}
+\sum_{n=1}^{N}\int_{\om}(\p_{n}\omega)\wedge*\,(\p_{n}\bar{\zeta})\\
&\;=\sum_{I}\big(\int_{\om}\omega_{I}\bar{\zeta}_{I}+
\sum_{n=1}^{N}\int_{\om}\p_{n}\omega_{I}\p_{n}\bar{\zeta}_{I}\big)
=\sum_{I}\big(\scpltom{\omega_{I}}{\zeta_{I}}
+\scpltom{\na\omega_{I}}{\na\zeta_{I}}\big)\\
&\;=\scpltom{\vec\omega}{\vec\zeta}
+\scpltom{\na\vec\omega}{\na\vec\zeta}
=\scp{\vec\omega}{\vec\zeta}_{\hoom}.
\end{align*}
Note that
$$\hoom=\Hoom{0}=\eDom{0}=*\,\cDom{N},\quad
\hocom=\Hocom{0}=\eDcom{0}=*\,\cDcom{N}$$
and 
$$\ed\omega=\sum_{n}\p_{n}\omega\ed x^{n},\qquad
\omega\in\hoom.$$

Like before, Weck's selection theorem (tangential version), stating that the embedding
\begin{align}
\mylabel{introallcompactembmaxtdf}
\eDcom{q}\cap\eps^{-1}\cDom{q}\cptemb\Ltom{q}
\end{align}
is compact, holds, see \cite{weckmax} for bounded strong Lipschitz domains (strong cone property)
and \cite{picardcomimb} for bounded weak Lipschitz domains.
The compact embeddings \eqref{introallcompactembmaxt}, \eqref{introallcompactembmaxn}
hold even for bounded weak Lipschitz domains and mixed boundary conditions,
see, e.g., the recent results \cite[Theorem 4.7, Theorem 4.8]{bauerpaulyschomburgmcpweaklip}.
The first proof of Weck's selection theorem \eqref{introallcompactembmaxtdf} 
for strong Lipschitz domains (strong/uniform cone property),
even for differential forms on Riemannian manifolds (and hence especially for $\om\subset\rN$),
has been given by Weck in \cite{weckmax}. 
The first proof for weak Lipschitz domains/manifolds 
is due to Picard and given in \cite{picardcomimb}. 
More related results and generalizations can be found in
\cite{leisbook,picardpotential,picardboundaryelectro,picarddeco,picardweckwitschxmas,
webercompmax,witschremmax,jochmanncompembmaxmixbc,goldshteinmitreairinamariushodgedecomixedbc,
jakabmitreairinamariusfinensolhodgedeco}.
Note that the boundedness of the underlying domain $\om$ is crucial,
since one has to work in polynomially weighted Sobolev spaces in unbounded (like exterior) domains,
see, e.g., \cite{kuhnpaulyregmax,leistheoem,leisbook,paulytimeharm,paulystatic,
paulydeco,paulyasym,picardpotential,picardweckwitschxmas}.

As we obtain the corresponding normal version
$$\eDom{q}\cap\eps^{-1}\cDcom{q}\cptemb\Ltom{q}$$
by applying the $*$-operator, see \eqref{starDdef},
we may concentrate on the tangential version \eqref{introallcompactembmaxtdf}.
Especially, \eqref{introallcompactembmaxtdf} implies (again by an indirect argument)
the following Maxwell type estimate:
There exists $\Cteps{q}>0$ such that for all 
$\omega\in\eDcom{q}\cap\eps^{-1}\cDom{q}\cap\Harmdepsom{q}^{\bot_{\Ltepsom{q}}}$
\begin{align}
\mylabel{intromaxestelecdf}
\norm{\omega}_{\Ltepsom{q}}
\leq\Cteps{q}\big(\norm{\ed\omega}_{\Ltom{q+1}}^2+\norm{\cd\eps\,\omega}_{\Ltom{q-1}}^2\big)^{\nicefrac{1}{2}}
\end{align}
holds. Here, we introduce the finite dimensional (again the unit ball is compact)
kernel space of (harmonic) Dirichlet forms by
$$\Harmdepsom{q}:=\eDczom{q}\cap\eps^{-1}\cDzom{q}.$$
Throughout this paper, as already mentioned, we assume that always the best possible constants are chosen, e.g.,
$\Cteps{q}>0$ are defined by
\begin{align}
\mylabel{intromaxestelecdfconstdef}
\frac{1}{\Cteps{q}^2}
:=\min_{\omega}\frac{\norm{\ed\omega}_{\Ltom{q+1}}^2+\norm{\cd\eps\,\omega}_{\Ltom{q-1}}^2}{\norm{\omega}_{\Ltepsom{q}}^2},
\end{align}
where the minimum is taken over $\eDcom{q}\cap\eps^{-1}\cDom{q}\cap\Harmdepsom{q}^{\bot_{\Ltepsom{q}}}$.

The main result of this paper is Theorem \ref{maintheo}, i.e., for convex $\om$ and for all $q$ it holds
\begin{align}
\mylabel{introconstestdf}
\frac{\cf}{\epsh}
\leq\Cteps{q}
\leq\cp\epsh,\qquad
\cp\leq\frac{\diam(\om)}{\pi}.
\end{align}
Corollary \ref{maintheostar} shows that in the case of the other (normal) boundary condition, 
where the boundary condition is placed on $\eps^{-1}\cDcom{q}$
and the corresponding constant is denoted by $\Cneps{q}$,
the same result holds for $\Cneps{q}$ as well.
Especially for $\eps=\id$ we have for all $q$
\begin{align}
\mylabel{introconstestiddf}
\cpc=\Ct{0}=\Cn{N}\leq\Ct{q},\Cn{q}\leq\Ct{N}=\Cn{0}=\cp\leq\frac{\diam(\om)}{\pi}.
\end{align}
Here and generally throughout this contribution, we skip the index $\eps$ in our notations
if the case $\eps=\id$ is considered.
We emphasize that \eqref{introconstestdf} 
not only generalizes \eqref{introconstest} to $N$-dimensions,
but even improves \eqref{introconstest} in $3$-dimensions to
\begin{align}
\mylabel{introconstesttwo}
\frac{\cpc}{\epsh}\leq\cmteps,\cmneps\leq\cp\epsh.
\end{align}

In Remark \ref{nonconvexdomrem} we will present a corresponding result for a certain class of non-convex domains,
so-called one-chart or one-map domains, which are bi-Lipschitz transformations of convex domains.
By a standard partition of unity argument we obtain results for general weak Lipschitz domains as well.

To prove our main result \eqref{introconstestdf} we will only use 
\begin{itemize}
\item
the well-known Friedrichs/Gaffney regularity and estimate 
for bounded and convex $\ci$-smooth domains $\om\subset\rN$, i.e.,
$\eDcom{q}\cap\cDom{q}$ and $\eDom{q}\cap\cDcom{q}$ are subspaces of $\Hoom{q}$ and 
\begin{align}
\mylabel{friedrichsgaffneysmoothconvex}
\forall\,\omega&\in\big(\eDcom{q}\cap\cDom{q}\big)\cup\big(\eDom{q}\cap\cDcom{q}\big)
&
\normltom{\na\vec\omega}^2
&\leq\norm{\ed\omega}_{\Ltom{q+1}}^2
+\norm{\cd\omega}_{\Ltom{q-1}}^2,
\end{align}
\item
Weck's selection theorem \eqref{introallcompactembmaxtdf},
which includes Rellich's selection theorems as special cases $q=0$ or $q=N$, 
\item
and some fundamental results from functional analysis.
\end{itemize}
For the regularity part of \eqref{friedrichsgaffneysmoothconvex}
see also \cite{kuhnpaulyregmax}.

Using vector proxies for the respective differential forms 
we get back the classical case of vector fields in $\rt$ or $\rN$
for the special choice $q=1$ or $q=N-1$.
Note that without using differential forms and vector proxies
$\curl E$ of a smooth vector field $E$ in $\rN$ 
may be defined point-wise as a vector in $\reals^{(N-1)N/2}$, which is
isomorphic to the skew-symmetric part of the Jacobian of $E$, i.e.,
$$\curl E\,\hat{=}\,2\skw\na E=\na E-(\na E)^{\top}\in\rNtN.$$
Finally, \eqref{introconstestdf} and \eqref{introconstestiddf}
hold for \eqref{intromaxestelec} and \eqref{intromaxestmag} in $\rN$ as well.

\section{Preliminaries}

Throughout this paper let $\om\subset\rN$, $N\geq2$, be a bounded weak Lipschitz domain.
Hence Weck's selection theorem \eqref{introallcompactembmaxtdf} and
the Maxwell type estimate \eqref{intromaxestelecdf} hold true.

\subsection{Functional Analysis Toolbox}
\label{fatsec}

Let $\A\!:\!D(\A)\subset\H_{1}\to\H_{2}$ denote a
closed and densely defined linear operator on two Hilbert spaces $\H_{1}$ and $\H_{2}$
with Hilbert space adjoint $\As\!:\!D(\As)\subset\H_{2}\to\H_{1}$.
Typically, $\A$ and $\As$ are unbounded. The adjoint is characterized by
\begin{align}
\mylabel{adjointchar}
\forall\,x\in D(\A)\quad
\forall\,y\in D(\As)\qquad
\scp{\A x}{y}_{\H_{2}}=\scp{x}{\As y}_{\H_{1}}.
\end{align}
Note $(\As)^{*}=\ovl{\A}=\A$, i.e., $(\A,\As)$ is a dual pair.
This shows the trivial but helpful result
\begin{align}
\mylabel{AssAb}
D(\A)=D\big((\As)^{*}\big)
=\setb{x\in\H_{1}}{\exists\,f\in\H_{2}\;\forall\,y\in D(\As)\quad\scp{x}{\As y}_{\H_{1}}=\scp{f}{y}_{\H_{2}}}.
\end{align}
By the projection theorem the Helmholtz type decompositions
\begin{align}
\mylabel{helm}
\H_{1}=N(\A)\oplus_{\H_{1}}\ovl{R(\As)},\qquad
\H_{2}=N(\As)\oplus_{\H_{2}}\ovl{R(\A)}
\end{align}
hold, where we introduce the notation $N$ for the kernel (or null space)
and $R$ for the range of a linear operator
and $\oplus_{\H}$ denotes orthogonality in a Hilbert space $\H$.
We define the reduced operators
\begin{align*}
\cA&:=\A|_{\ovl{R(\As)}}:D(\cA)\subset\ovl{R(\As)}\to\ovl{R(\A)},&
D(\cA)&:=D(\A)\cap N(\A)^{\bot_{\H_{1}}}=D(\A)\cap\ovl{R(\As)},\\
\cAs&:=\As|_{\ovl{R(\A)}}:D(\cAs)\subset\ovl{R(\A)}\to\ovl{R(\As)},&
D(\cAs)&:=D(\As)\cap N(\As)^{\bot_{\H_{2}}}=D(\As)\cap\ovl{R(\A)},
\end{align*}
which are also closed and densely defined linear operators.
We note that $\cA$ and $\cAs$ are indeed adjoint to each other, i.e.,
$(\cA,\cAs)$ is a dual pair as well. Now the inverse operators 
$$\cA^{-1}:R(\A)\to D(\cA),\qquad
(\cAs)^{-1}:R(\As)\to D(\cAs)$$
exist and they are bijective, 
since $\cA$ and $\cAs$ are injective by definition.
Furthermore, by \eqref{helm} we have
the refined Helmholtz type decompositions
\begin{align}
\label{DacA}
D(\A)&=N(\A)\oplus_{\H_{1}}D(\cA),&
D(\As)&=N(\As)\oplus_{\H_{2}}D(\cAs)
\intertext{and thus we obtain for the ranges}
\label{RacA}
R(\A)&=R(\cA),&
R(\As)&=R(\cAs).
\end{align}

Using the closed range theorem and the closed graph theorem we get the following result.

\begin{lem}
\label{poincarerange}
The following assertions are equivalent:
\begin{itemize}
\item[\bf(i)] 
$\exists\,c_{\A}\,\,\in(0,\infty)$ \quad 
$\forall\,x\in D(\cA)$ \qquad
$\norm{x}_{\H_{1}}\leq c_{\A}\norm{\A x}_{\H_{2}}$
\item[\bf(i${}^{*}$)] 
$\exists\,c_{\As}\in(0,\infty)$ \quad 
$\forall\,y\in D(\cAs)$ \quad\;
$\norm{y}_{\H_{2}}\leq c_{\As}\norm{\As y}_{\H_{1}}$
\item[\bf(ii)] 
$R(\A)=R(\cA)$ is closed in $\H_{2}$.
\item[\bf(ii${}^{*}$)] 
$R(\As)=R(\cAs)$ is closed in $\H_{1}$.
\item[\bf(iii)] 
$\cA^{-1}:R(\A)\to D(\cA)$ is continuous and bijective
with norm bounded by $(1+c_{\A}^2)^{\oh}$.
\item[\bf(iii${}^{*}$)] 
$(\cAs)^{-1}:R(\As)\to D(\cAs)$ is continuous and bijective
with norm bounded by $(1+c_{\As}^2)^{\oh}$.
\end{itemize}
If one of these assertions holds true, e.g., (ii), $R(\A)=R(\cA)$ is closed, then 
\begin{align*}
\cA:D(\cA)
&\subset R(\As)\to R(\A),
&
D(\cA)
&=D(\A)\cap R(\As),\\
\cAs:D(\cAs)
&\subset R(\A)\to R(\As),
&
D(\cAs)
&=D(\As)\cap R(\A),
\intertext{and the Helmholtz type decompositions}
\H_{1}
&=N(\A)\oplus_{\H_{1}}R(\As),
&
\H_{2}
&=N(\As)\oplus_{\H_{2}}R(\A),\\
D(\A)
&=N(\A)\oplus_{\H_{1}}D(\cA),
&
D(\As)
&=N(\As)\oplus_{\H_{2}}D(\cAs)
\end{align*}
hold.
\end{lem}

Throughout this paper we will assume that always the ``best'' Friedrichs/Poincar\'e type constants are chosen, i.e.,
$c_{\A},c_{\As}\in(0,\infty]$ are given by the usual Rayleigh quotients 
$$\frac{1}{c_{\A}}
:=\inf_{0\neq x\in D(\cA)}\frac{\norm{\A x}_{\H_{2}}}{\norm{x}_{\H_{1}}},\qquad
\frac{1}{c_{\As}}
:=\inf_{0\neq y\in D(\cAs)}\frac{\norm{\As y}_{\H_{1}}}{\norm{y}_{\H_{2}}}.$$

\begin{lem}
\label{lemconstants}
The Friedrichs/Poincar\'e type constants coincide, i.e., $c_{\A}=c_{\As}\in(0,\infty]$.
\end{lem}

\begin{lem}
\mylabel{cptequi}
The following assertions are equivalent:
\begin{itemize}
\item[\bf(i)] 
$D(\cA)\cptemb\H_{1}$ is compact.
\item[\bf(i${}^{*}$)] 
$D(\cAs)\cptemb\H_{2}$ is compact.
\item[\bf(ii)] 
$\cA^{-1}:R(\A)\to R(\As)$ is compact
with norm $c_{\A}$.
\item[\bf(ii${}^{*}$)] 
$(\cAs)^{-1}:R(\As)\to R(\A)$ is compact
with norm $c_{\As}=c_{\A}$.
\end{itemize}
If one of these assertions holds true, e.g., (i), $D(\cA)\cptemb\H_{1}$ is compact,
then (by a standard indirect argument showing Lemma \ref{poincarerange} (i)) the assertions of the latter two lemmas hold.
Especially, the Friedrichs/Poincar\'e type estimates hold,
all ranges are closed and the inverse operators are compact.
\end{lem}

Now, let $\Az\!:\!D(\Az)\subset\H_{0}\to\H_{1}$ and $\Ao\!:\!D(\Ao)\subset\H_{1}\to\H_{2}$ 
be (possibly unbounded) closed and densely defined linear operators 
on three Hilbert spaces $\H_{0}$, $\H_{1}$, and $\H_{2}$
with adjoints $\Azs\!:\!D(\Azs)\subset\H_{1}\to\H_{0}$
and $\Aos\!:\!D(\Aos)\subset\H_{2}\to\H_{1}$
as well as reduced operators $\cAz$, $\cAzs$, and $\cAo$, $\cAos$.
Furthermore, we assume the sequence or complex property of $\Az$ and $\Ao$, 
that is, $\Ao\Az\subset0$, i.e.,
\begin{align}
\mylabel{sequenceprop}
R(\Az)\subset N(\Ao).
\end{align}
Then also $\Azs\Aos\subset0$, i.e., $R(\Aos)\subset N(\Azs)$,
as for all $x\in D(\Az)$, $y\in R(\Aos)$ with $y=\Aos z$, $z\in D(\Aos)$
$$\scp{y}{\Az x}_{\H_{1}}
=\scp{\Aos z}{\Az x}_{\H_{1}}
=\scp{z}{\Ao\Az x}_{\H_{2}}
=0.$$
The Helmholtz type decompositions \eqref{helm} for $\A=\Az$ and $\A=\Ao$ read, e.g.,
\begin{align}
\label{helmappclone}
\H_{1}&=\ovl{R(\Az)}\oplus_{\H_{1}}N(\Azs),
&
\H_{1}&=N(\Ao)\oplus_{\H_{1}}\ovl{R(\Aos)},
\intertext{and by the complex properties \eqref{sequenceprop} we obtain}
\nonumber
D(\Ao)&=\ovl{R(\Az)}\oplus_{\H_{1}}\big(D(\Ao)\cap N(\Azs)\big),
&
D(\Azs)&=\big(D(\Azs)\cap N(\Ao)\big)\oplus_{\H_{1}}\ovl{R(\Aos)},\\
\nonumber
N(\Ao)&=\ovl{R(\Az)}\oplus_{\H_{1}}N_{0,1},
&
N(\Azs)&=N_{0,1}\oplus_{\H_{1}}\ovl{R(\Aos)},
\end{align}
where we define the cohomology group
$$N_{0,1}:=N(\Ao)\cap N(\Azs).$$
Putting things together, the general refined Helmholtz type decomposition
\begin{align}
\label{helmappfullcl}
\H_{1}=\ovl{R(\Az)}\oplus_{\H_{1}}N_{0,1}\oplus_{\H_{1}}\ovl{R(\Aos)},\qquad
R(\Az)=R(\cAz),\quad
R(\Aos)=R(\cAos)
\end{align}
holds. The previous results of this section imply immediately the following.

\begin{lem}
\label{helmrefined}
Let $\Az$, $\Ao$ be as introduced before with $\Ao\Az\subset0$, i.e., \eqref{sequenceprop}.
Moreover, let $R(\Az)$ and $R(\Ao)$ be closed.
Then, the assertions of Lemma \ref{poincarerange} and Lemma \ref{lemconstants}
hold for $\Az$ and $\Ao$. Moreover, the refined Helmholtz type decompositions
\begin{align*}
\H_{1}
&=R(\Az)\oplus_{\H_{1}}N_{0,1}\oplus_{\H_{1}}R(\Aos),
\\
N(\Ao)
&=R(\Az)\oplus_{\H_{1}}N_{0,1},
&
N(\Azs)
&=N_{0,1}\oplus_{\H_{1}}R(\Aos),\\
D(\Ao)
&=R(\Az)\oplus_{\H_{1}}N_{0,1}\oplus_{\H_{1}}D(\cAo),
&
D(\Azs)
&=D(\cAzs)\oplus_{\H_{1}}N_{0,1}\oplus_{\H_{1}}R(\Aos),\\
D(\Ao)\cap D(\Azs)
&=D(\cAzs)\oplus_{\H_{1}}N_{0,1}\oplus_{\H_{1}}D(\cAo)
\end{align*}
hold. Especially, 
\begin{align*}
R(\Az)
&=N(\Ao)\cap N_{0,1}^{\bot_{\H_{1}}},
&
R(\Azs)
&,
&
R(\Ao)
&,
&
R(\Aos)
&=N(\Azs)\cap N_{0,1}^{\bot_{\H_{1}}}
\end{align*} 
are closed, the respective inverse operators, i.e.,
\begin{align*}
\cAz^{-1}&:R(\Az)\to D(\cAz),
&
\cAo^{-1}&:R(\Ao)\to D(\cAo),\\
(\cAzs)^{-1}&:R(\Azs)\to D(\cAzs),
&
(\cAos)^{-1}&:R(\Aos)\to D(\cAos),
\end{align*}
are continuous, and there exist positive constants $c_{\Az}$, $c_{\Ao}$,
such that the Friedrichs/Poincar\'e type estimates
\begin{align*}
\forall\,x&\in D(\cAz)
&
\norm{x}_{\H_{0}}&\leq c_{\Az}\norm{\Az x}_{\H_{1}},
&
\forall\,y&\in D(\cAo)
&
\norm{y}_{\H_{1}}&\leq c_{\Ao}\norm{\Ao y}_{\H_{2}},\\
\forall\,y&\in D(\cAzs)
&
\norm{y}_{\H_{1}}&\leq c_{\Az}\norm{\Azs y}_{\H_{0}},
&
\forall\,z&\in D(\cAos)
&
\norm{z}_{\H_{2}}&\leq c_{\Ao}\norm{\Aos z}_{\H_{1}}
\end{align*}
hold.
\end{lem}

\begin{rem}
\label{clrangecompembrem}
If, e.g., $D(\cAz)\cptemb\H_{0}$ and $D(\cAo)\cptemb\H_{1}$ are compact,
then $R(\Az)$ and $R(\Ao)$ are closed and hence the assertions of Lemma \ref{helmrefined} hold.
Moreover, the respective inverse operators, i.e.,
\begin{align*}
\cAz^{-1}&:R(\Az)\to R(\Azs),
&
\cAo^{-1}&:R(\Ao)\to R(\Aos),\\
(\cAzs)^{-1}&:R(\Azs)\to R(\Az),
&
(\cAos)^{-1}&:R(\Aos)\to R(\Ao),
\end{align*}
are compact.
\end{rem}

By the complex property we observe
$D(\cAo),D(\cAzs)\subset D(\Ao)\cap D(\Azs)$.
Utilizing the Helmholtz type decomposition \eqref{helmappfullcl} we immediately see the following. 

\begin{lem}
\label{compemblem}
The embeddings 
$D(\cAz)\cptemb\H_{0}$, $D(\cAo)\cptemb\H_{1}$,
and $N_{0,1}\cptemb\H_{1}$ are compact, if and only if the embedding
$D(\Ao)\cap D(\Azs)\cptemb\H_{1}$ is compact.
In this case, $N_{0,1}$ has finite dimension.
\end{lem}

\begin{rem}
\label{complexone}
Let us consider the sequence or complex
\begin{align}
\mylabel{primalcomplex}
\begin{CD}
D(\Az)\subset\H_{0} @> \Az >>
D(\Ao)\subset\H_{1} @> \Ao >>
\H_{2}.
\end{CD}
\end{align}
\begin{itemize}
\item[\bf(i)]
The general assumptions on $\Az$ and $\Ao$ are equivalent to the assumption that
\eqref{primalcomplex} is a Hilbert complex, meaning that the operators are closed 
and satisfy the complex property \eqref{sequenceprop}. 
\item[\bf(ii)]
The assumption that the ranges $R(\Az)$ and $R(\Ao)$ are closed is equivalent to the assumption that
\eqref{primalcomplex} is a closed Hilbert complex. 
\item[\bf(iii)]
The assumption that the embeddings $D(\cAz)\cptemb\H_{0}$ and $D(\cAo)\cptemb\H_{1}$ are compact 
is equivalent to the assumption that \eqref{primalcomplex} is a compact Hilbert complex,
which is always closed. 
\item[\bf(iv)]
The assumption that the embedding $D(\Ao)\cap D(\Azs)\cptemb\H_{1}$ is compact is equivalent to the assumption that
\eqref{primalcomplex} is a Fredholm complex, meaning that the complex is compact 
and the cohomology group $N_{0,1}$ is finite dimensional.
\end{itemize}
The strongest property (iv) is the most desirable one,
and we can realize this is our applications.
By the previous results, any property of the primal complex \eqref{primalcomplex} 
is transferred to the corresponding property of the dual complex
$$\begin{CD}
\H_{0} @< \Azs <<
D(\Azs)\subset\H_{1} @< \Aos <<
D(\Aos)\subset\H_{2}
\end{CD}$$
and vise verse.
\end{rem}

We can summarize.

\begin{theo}
\label{fatbmaintheogen}
Let $\Az$, $\Ao$ be as introduced, i.e., having the complex property $R(\Az)\subset N(\Ao)$. 
Moreover, let $D(\Ao)\cap D(\Azs)\cptemb\H_{1}$ be compact.
Then the assertions of Lemma \ref{helmrefined} hold, $N_{0,1}$ is finite dimensional
and the corresponding inverse operators are continuous resp. compact.
Especially, all ranges are closed and the corresponding Friedrichs/Poincar\'e type estimates hold.
\end{theo}

\begin{theo}
\label{constesthilbert}
Let $\Az$, $\Ao$ be as introduced, i.e., having the complex property $R(\Az)\subset N(\Ao)$,
and let $D(\Ao)\cap D(\Azs)\cptemb\H_{1}$ be compact. Then
$$\forall\,x\in D(\Ao)\cap D(\Azs)\cap N_{0,1}^{\perp_{\H_{1}}}\qquad
\norm{x}_{\H_{1}}^2
\leq c_{\Az}^2\norm{\Azs x}_{\H_{0}}^2
+c_{\Ao}^2\norm{\Ao x}_{\H_{2}}^2.$$
Especially,
$$\forall\,x\in D(\Ao)\cap D(\Azs)\cap N_{0,1}^{\perp_{\H_{1}}}\qquad
\norm{x}_{\H_{1}}
\leq\max\{c_{\Az},c_{\Ao}\}
\big(\norm{\Azs x}_{\H_{0}}^2
+\norm{\Ao x}_{\H_{2}}^2\big)^{\oh}.$$
\end{theo}

\begin{proof}
Let $x\in D(\Ao)\cap D(\Azs)\cap N_{0,1}^{\perp_{\H_{1}}}$.
By the Helmholtz type decomposition of Lemma \ref{helmrefined} we have 
$$D(\Ao)\cap D(\Azs)\cap N_{0,1}^{\perp_{\H_{1}}}=D(\cAzs)\oplus_{\H_{1}}D(\cAo)$$
and hence we can decompose
$$x=x_{0}+x_{1}\in D(\cAzs)\oplus_{\H_{1}}D(\cAo),\qquad
\Azs x=\Azs x_{0},\quad
\Ao x=\Ao x_{1}.$$
By orthogonality and the Friedrichs/Poincar\'e type estimates we get
\begin{align*}
\norm{x}_{\H_{1}}^2
=\norm{x_{0}}_{\H_{1}}^2
+\norm{x_{1}}_{\H_{1}}^2
\leq c_{\Az}^2\norm{\Azs x_{0}}_{\H_{0}}^2
+c_{\Ao}^2\norm{\Ao x_{1}}_{\H_{2}}^2
=c_{\Az}^2\norm{\Azs x}_{\H_{0}}^2
+c_{\Ao}^2\norm{\Ao x}_{\H_{2}}^2,
\end{align*}
completing the proof.
\end{proof}

\begin{rem}
\label{constesthilbertrem}
In Theorem \ref{constesthilbert}
$\max\{c_{\Az},c_{\Ao}\}=c_{\Az,\Ao}$ is the best constant (or sharp), where
$$\frac{1}{c_{\Az,\Ao}^2}
:=\inf_{0\neq x\in D(\Ao)\cap D(\Azs)\cap N_{0,1}^{\perp_{\H_{1}}}}
\frac{\norm{\Azs x}_{\H_{0}}^2+\norm{\Ao x}_{\H_{2}}^2}{\norm{x}_{\H_{1}}^2}.$$
It is clear that $c_{\Az,\Ao}\leq\max\{c_{\Az},c_{\Ao}\}$ holds by Theorem \ref{constesthilbert}.
On the other hand, looking at the subspaces (ranges) of the Helmholtz type decompositions
one obtains immediately $c_{\Az}\leq c_{\Az,\Ao}$, if , e.g., $\max\{c_{\Az},c_{\Ao}\}=c_{\Az}$.
\end{rem}

\subsection{Applications to Differential Forms}

We will apply Theorem \ref{constesthilbert} in our differential form setting.
As closure of the exterior derivative defined on $\Cicom{q}$ 
as an unbounded operator on $\ltom$ we get that 
$$\edc_{q}:\eDcom{q}\subset\Ltom{q}\to\Ltom{q+1}$$
is a closed and densely defined linear operator with closed adjoint
$$\edc_{q}^{*}=\cd_{q+1}:\cDom{q+1}\subset\Ltom{q+1}\to\Ltom{q}.$$
These operators satisfy the natural complex property $\edc_{q+1}\edc_{q}\subset0$, i.e.,
$R(\edc_{q})\subset N(\edc_{q+1})$, and thus also
$\cd_{q}\cd_{q+1}\subset0$, i.e.,
$R(\cd_{q+1})\subset N(\cd_{q})$. 
Analogously or using the $*$-operator we can define 
closed operators for the other boundary condition, i.e.,
$$\ed_{q}:\eDom{q}\subset\Ltom{q}\to\Ltom{q+1},\quad
\ed_{q}^{*}=\cdc_{q+1}:\cDcom{q+1}\subset\Ltom{q+1}\to\Ltom{q},$$
which also satisfy the complex properties, i.e.,
$\ed_{q+1}\ed_{q}\subset0$ and $\cdc_{q}\cdc_{q+1}\subset0$. Note that
\begin{align*}
D(\edc_{q})
&=\eDcom{q},
&
D(\ed_{q})
&=\eDom{q},
&
D(\cdc_{q})
&=\cDcom{q},
&
D(\cd_{q})
&=\cDom{q},\\
N(\edc_{q})
&=\eDczom{q},
&
N(\ed_{q})
&=\eDzom{q},
&
N(\cdc_{q})
&=\cDczom{q},
&
N(\cd_{q})
&=\cDzom{q}.
\end{align*}
By \eqref{adjointchar} we get trivially the rules of partial integration, i.e.,
\begin{align}
\mylabel{partintdifff}
\begin{split}
\forall\,\omega\in\eDcom{q}\quad
\forall\,\zeta\in\cDom{q+1}\qquad
\scp{\edc_{q}\omega}{\zeta}_{\Ltom{q+1}}
&=-\scp{\omega}{\cd_{q+1}\zeta}_{\Ltom{q}},\\
\forall\,\omega\in\eDom{q}\quad
\forall\,\zeta\in\cDcom{q+1}\qquad
\scp{\ed_{q}\omega}{\zeta}_{\Ltom{q+1}}
&=-\scp{\omega}{\cdc_{q+1}\zeta}_{\Ltom{q}}.
\end{split}
\end{align}
\eqref{AssAb} provides a useful characterization of homogeneous boundary conditions, i.e.,
\begin{align*}
\eDcom{q}
&=D(\edc_{q})
=D\big((\edc_{q}^{*})^{*}\big)
=D(\cd_{q+1}^{*})\\
&=\setb{\omega\in\Ltom{q}}{\exists\,\zeta\in\Ltom{q+1}\;\forall\,\varphi\in D(\cd_{q+1})=\cDom{q+1}\quad
\scp{\omega}{\cd_{q+1}\varphi}_{\Ltom{q}}=\scp{\zeta}{\varphi}_{\Ltom{q+1}}}\\
&=\setb{\omega\in\eDom{q}}{\forall\,\varphi\in\cDom{q+1}\quad
\scp{\omega}{\cd_{q+1}\varphi}_{\Ltom{q}}=\scp{\ed_{q}\omega}{\varphi}_{\Ltom{q+1}}},
\end{align*}
and analogously or by the $*$-operator we also get
\begin{align}
\mylabel{cDccharac}
\cDcom{q}
&=\setb{\omega\in\Ltom{q}}{\exists\,\xi\in\Ltom{q-1}\;\forall\,\varphi\in\eDom{q-1}\quad
\scp{\omega}{\ed_{q-1}\varphi}_{\Ltom{q}}=\scp{\xi}{\varphi}_{\Ltom{q-1}}}.
\end{align}

In the following we will skip the index $q$ on the operators 
and write just $\edc$, $\ed$ and $\cdc$, $\cd$. 
To incorporate the material law $\eps$ we need to modify these operators slightly.
For this, let us fix some $q=0,\dots,N$ and 
let $\eps$ be an admissible transformation on $q$-forms. Defining
the closed and densely defined linear operators
\begin{align*}
\Az:=\edc:\eDcom{q-1}
&\subset\Ltom{q-1}\to\Ltepsom{q},
&
\Ao:=\edc:\eDcom{q}
&\subset\Ltepsom{q}\to\Ltom{q+1},
\intertext{we see that their closed adjoints are}
\Azs=\edc^{*}=\cd\eps:\eps^{-1}\cDom{q}
&\subset\Ltepsom{q}\to\Ltom{q-1},
&
\Aos=\edc^{*}=\eps^{-1}\cd:\cDom{q+1}
&\subset\Ltom{q+1}\to\Ltepsom{q}.
\end{align*}
Again these operators satisfy the complex property $\Ao\Az=\edc\edc\subset0$, i.e.,
$R(\edc)\subset N(\edc)$, and thus also
$\Azs\Aos=\cd\eps\,\eps^{-1}\cd\subset0$, i.e.,
$R(\eps^{-1}\cd)\subset N(\cd\eps)$. As before, analogously or using the $*$-operator 
we can also define the closed operators
\begin{align*}
\tilde\A_{0}:=\ed:\eDom{q-1}
&\subset\Ltom{q-1}\to\Ltepsom{q},
&
\tilde\A_{1}:=\ed:\eDom{q}
&\subset\Ltepsom{q}\to\Ltom{q+1},\\
\tilde\A_{0}^{*}=\ed^{*}=\cdc\eps:\eps^{-1}\cDcom{q}
&\subset\Ltepsom{q}\to\Ltom{q-1},
&
\tilde\A_{1}^{*}=\ed^{*}=\eps^{-1}\cdc:\cDcom{q+1}
&\subset\Ltom{q+1}\to\Ltepsom{q},
\end{align*}
which satisfy the complex properties as well. 

We will focus on the operators
$\Az$, $\Ao$, $\Azs$, $\Aos$. 
At this point let us note that all results of the Functional Analysis Toolbox Section \ref{fatsec}
are applicable since by Weck's selection theorem \eqref{introallcompactembmaxtdf} the embedding
$$D(\Ao)\cap D(\Azs)=\eDcom{q}\cap\eps^{-1}\cDom{q}\cptemb\Ltepsom{q}=\H_{1}$$
is compact, see, e.g., Theorem \ref{fatbmaintheogen}.
Especially, all ranges are closed, the inverse operators are continuous resp. compact,
the corresponding Friedrichs/Poincar\'e type estimates 
and Helmholtz type decompositions hold, and the cohomology group
$$N_{0,1}=N(\Ao)\cap N(\Azs)=\eDczom{q}\cap\eps^{-1}\cDzom{q}=\Harmdepsom{q}$$
has finite dimension. The corresponding reduced operators are
\begin{align*}
\cAz=\edc:\eDcom{q-1}\cap\cd\cDom{q}
&\subset\cd\cDom{q}\to\edc\eDcom{q-1},\\
\cAzs=\cd\eps:\eps^{-1}\cDom{q}\cap\edc\eDcom{q-1}
&\subset\edc\eDcom{q-1}\to\cd\cDom{q},\\
\cAo=\edc:\eDcom{q}\cap\eps^{-1}\cd\cDom{q+1}
&\subset\eps^{-1}\cd\cDom{q+1}\to\edc\eDcom{q},\\
\cAos=\eps^{-1}\cd:\cDom{q+1}\cap\edc\eDcom{q}
&\subset\edc\eDcom{q}\to\eps^{-1}\cd\cDom{q+1},
\end{align*}
where $\edc\eDcom{q-1}$ and $\eps^{-1}\cd\cDom{q+1}$ 
have to be understood as closed subspaces of $\Ltepsom{q}$.
In this case, Lemma \ref{helmrefined} and Theorem \ref{fatbmaintheogen} read as follows.

\begin{cor}
\label{helmrefinedcor}
The refined Helmholtz type decompositions
\begin{align*}
\Ltepsom{q}
&=\edc\eDcom{q-1}\oplus_{\Ltepsom{q}}\Harmdepsom{q}\oplus_{\Ltepsom{q}}\eps^{-1}\cd\cDom{q+1},\\
\eDczom{q}
&=\edc\eDcom{q-1}\oplus_{\Ltepsom{q}}\Harmdepsom{q},\\
\eps^{-1}\cDzom{q}
&=\Harmdepsom{q}\oplus_{\Ltepsom{q}}\eps^{-1}\cd\cDom{q+1},\\
\eDcom{q}
&=\edc\eDcom{q-1}\oplus_{\Ltepsom{q}}\Harmdepsom{q}
\oplus_{\Ltepsom{q}}\big(\eDcom{q}\cap\eps^{-1}\cd\cDom{q+1}\big),\\
\eps^{-1}\cDom{q}
&=\big(\eps^{-1}\cDom{q}\cap\edc\eDcom{q-1}\big)
\oplus_{\Ltepsom{q}}\Harmdepsom{q}\oplus_{\Ltepsom{q}}\eps^{-1}\cd\cDom{q+1},\\
\eDcom{q}\cap\eps^{-1}\cDom{q}
&=\big(\eps^{-1}\cDom{q}\cap\edc\eDcom{q-1}\big)
\oplus_{\Ltepsom{q}}\Harmdepsom{q}\oplus_{\Ltepsom{q}}\big(\eDcom{q}\cap\eps^{-1}\cd\cDom{q+1}\big)
\end{align*}
hold, all ranges
\begin{align*}
\eDczom{q}\cap\Harmdepsom{q}^{\bot_{\Ltepsom{q}}}
=\edc\eDcom{q-1}
&=\edc\big(\eDcom{q-1}\cap\cd\cDom{q}\big),\\
\edc\eDcom{q}
&=\edc\big(\eDcom{q}\cap\eps^{-1}\cd\cDom{q+1}\big),\\
\cd\cDom{q}
&=\cd\big(\eps^{-1}\cDom{q}\cap\edc\eDcom{q-1}\big),\\
\eps^{-1}\cDzom{q}\cap\Harmdepsom{q}^{\bot_{\Ltepsom{q}}}
=\eps^{-1}\cd\cDom{q+1}
&=\eps^{-1}\cd\big(\cDom{q+1}\cap\edc\eDcom{q}\big)
\end{align*}
are closed, the space of Dirichlet forms $\Harmdepsom{q}=\eDczom{q}\cap\eps^{-1}\cDzom{q}$ is finite dimensional,
the respective inverse operators, i.e.,
\begin{align*}
\cAz^{-1}=\edc^{-1}:\edc\eDcom{q-1}&\to\eDcom{q-1}\cap\cd\cDom{q},\\
\cAo^{-1}=\edc^{-1}:\edc\eDcom{q}&\to\eDcom{q}\cap\eps^{-1}\cd\cDom{q+1},\\
(\cAzs)^{-1}=(\cd\eps)^{-1}:\cd\cDom{q}&\to\big(\eps^{-1}\cDom{q}\cap\edc\eDcom{q-1}\big),\\
(\cAos)^{-1}=(\eps^{-1}\cd)^{-1}:\eps^{-1}\cd\cDom{q+1}&\to\cDom{q+1}\cap\edc\eDcom{q},
\end{align*}
are continuous, and there exist positive constants $c_{\Az}=\tCedcteps{q-1}$ and $c_{\Ao}=\Cedcteps{q}$,
such that the Friedrichs/Poincar\'e type estimates
\begin{align*}
\forall\,\xi&\in\eDcom{q-1}\cap\cd\cDom{q}
&
\norm{\xi}_{\Ltom{q-1}}&\leq\tCedcteps{q-1}\norm{\edc\xi}_{\Ltepsom{q}},\\
\forall\,\omega&\in\eDcom{q}\cap\eps^{-1}\cd\cDom{q+1}
&
\norm{\omega}_{\Ltepsom{q}}&\leq\Cedcteps{q}\norm{\edc\omega}_{\Ltom{q+1}},\\
\forall\,\omega&\in\eps^{-1}\cDom{q}\cap\edc\eDcom{q-1}
&
\norm{\omega}_{\Ltepsom{q}}&\leq\tCedcteps{q-1}\norm{\cd\eps\,\omega}_{\Ltom{q-1}},\\
\forall\,\zeta&\in\cDom{q+1}\cap\edc\eDcom{q}
&
\norm{\zeta}_{\Ltom{q+1}}&\leq\Cedcteps{q}\norm{\eps^{-1}\cd\zeta}_{\Ltepsom{q}}
\end{align*}
hold.
\end{cor}

\begin{rem}
\label{helmrefinedcorremone}
The corresponding corollary holds for the other boundary conditions on $\cDcom{\dots}$
for the operators $\tilde\A_{0}$, $\tilde\A_{0}^{*}$, $\tilde\A_{1}$, $\tilde\A_{1}^{*}$ as well.
\end{rem}

For $\eps=\id$ just one constant for a single $q$ is needed. More precisely:

\begin{lem}
\label{helmrefinedcorlem}
Let $\eps=\id$. Then for all $q$
$$\tCedct{q}=\Cedct{q}$$
and the Friedrichs/Poincar\'e type estimates
\begin{align*}
\forall\,\omega&\in\eDcom{q}\cap\cd\cDom{q+1}
&
\norm{\omega}_{\Ltom{q}}&\leq\Cedct{q}\norm{\edc\omega}_{\Ltom{q+1}},\\
\forall\,\zeta&\in\cDom{q+1}\cap\edc\eDcom{q}
&
\norm{\zeta}_{\Ltom{q+1}}&\leq\Cedct{q}\norm{\cd\zeta}_{\Ltom{q}}
\end{align*}
hold. Applying the $*$-operator we have
\begin{align*}
\forall\,\omega&\in\cDcom{N-q}\cap\ed\eDom{N-q-1}
&
\norm{\omega}_{\Ltom{N-q}}&\leq\Cedct{q}\norm{\cdc\omega}_{\Ltom{N-q-1}},\\
\forall\,\zeta&\in\eDom{N-q-1}\cap\cdc\cDcom{N-q}
&
\norm{\zeta}_{\Ltom{N-q-1}}&\leq\Cedct{q}\norm{\ed\zeta}_{\Ltom{N-q}}.
\end{align*}
All these four Friedrichs/Poincar\'e type estimates
hold with the same best constants $\Cedct{q}$.
\end{lem}

With these settings our estimate of interest \eqref{intromaxestelecdf}, i.e., 
$$\norm{\omega}_{\Ltepsom{q}}
\leq\Cteps{q}\big(\norm{\edc\omega}_{\Ltom{q+1}}^2+\norm{\cd\eps\,\omega}_{\Ltom{q-1}}^2\big)^{\nicefrac{1}{2}}$$
for all $\omega\in\eDcom{q}\cap\eps^{-1}\cDom{q}\cap\Harmdepsom{q}^{\bot_{\Ltepsom{q}}}$, reads
$$\forall\,x\in D(\Ao)\cap D(\Azs)\cap N_{0,1}^{\perp_{\H_{1}}}\qquad
\norm{x}_{\H_{1}}
\leq c_{\Az,\Ao} 
\big(\norm{\Ao x}_{\H_{2}}^2
+\norm{\Azs x}_{\H_{0}}^2\big)^{\oh}$$
and by Theorem \ref{constesthilbert} and Remark \ref{constesthilbertrem} we know
$$\Cteps{q}=c_{\Az,\Ao}=\max\{c_{\Az},c_{\Ao}\}=\max\{\tCedcteps{q-1},\Cedcteps{q}\}$$
using the notations from Corollary \ref{helmrefinedcor}. More precisely, Theorem \ref{constesthilbert} shows:

\begin{cor}
\label{constesthilbertdf}
For all $\omega\in\eDcom{q}\cap\eps^{-1}\cDom{q}\cap\Harmdepsom{q}^{\bot_{\Ltepsom{q}}}$
$$\norm{\omega}_{\Ltepsom{q}}^2
\leq\Cedcteps{q}^2\norm{\edc\omega}_{\Ltom{q+1}}^2
+\tCedcteps{q-1}^2\norm{\cd\eps\,\omega}_{\Ltom{q-1}}^2$$
and hence
$$\norm{\omega}_{\Ltepsom{q}}
\leq\Cteps{q}\big(\norm{\edc\omega}_{\Ltom{q+1}}^2+\norm{\cd\eps\,\omega}_{\Ltom{q-1}}^2\big)^{\nicefrac{1}{2}},\qquad
\Cteps{q}=\max\{\tCedcteps{q-1},\Cedcteps{q}\}.$$
\end{cor}

\section{Main Results}

By Corollary \ref{constesthilbertdf} we have to find upper and lower bounds for 
the constants $\tCedcteps{q-1}$ and $\Cedcteps{q}$.
As a first step, we take care of the dependencies on the transformation $\eps$.

\begin{lem}
\label{constestremoveeps}
It holds
$$\frac{\Cedct{q-1}}{\epso}\leq\tCedcteps{q-1}\leq\Cedct{q-1}\epsu,\qquad
\frac{\Cedct{q}}{\epsu}\leq\Cedcteps{q}\leq\Cedct{q}\epso.$$
Moreover, 
$$\frac{\min\{\Cedct{q-1},\Cedct{q}\}}{\epsh}
\leq\Cteps{q}
=\max\{\tCedcteps{q-1},\Cedcteps{q}\}
\leq\max\{\Cedct{q-1},\Cedct{q}\}\epsh.$$
\end{lem}

\begin{proof}
Let $\xi\in\eDcom{q-1}\cap\cd\cDom{q}$.
By Lemma \ref{helmrefinedcorlem} and \eqref{epsuoonedf}, \eqref{epsuotwodf} we see
$$\norm{\xi}_{\Ltom{q-1}}
\leq\Cedct{q-1}\norm{\edc\xi}_{\Ltom{q}}
\leq\Cedct{q-1}\epsu\,\norm{\edc\xi}_{\Ltepsom{q}},$$
and hence $\tCedcteps{q-1}\leq\Cedct{q-1}\epsu$. On the other hand, by 
Corollary \ref{helmrefinedcor} and \eqref{epsuoonedf}, \eqref{epsuotwodf}
$$\norm{\xi}_{\Ltom{q-1}}
\leq\tCedcteps{q-1}\norm{\edc\xi}_{\Ltepsom{q}}
\leq\tCedcteps{q-1}\epso\,\norm{\edc\xi}_{\Ltom{q}}$$
holds, and hence by Lemma \ref{helmrefinedcorlem} 
$\Cedct{q-1}\leq\tCedcteps{q-1}\epso$.
Now, pick $\omega\in\eDcom{q}\cap\eps^{-1}\cd\cDom{q+1}$.
According to Corollary \ref{helmrefinedcor} (with $\eps=\id$) it holds
$$\eDcom{q}
=\eDczom{q}\oplus_{\Ltom{q}}\big(\eDcom{q}\cap\cd\cDom{q+1}\big)$$
and we can decompose
$$\omega=\omega_{0}+\omega_{\cd},\qquad
\omega_{0}\in\eDczom{q},\quad
\omega_{\cd}\in\eDcom{q}\cap\cd\cDom{q+1}$$
with $\edc\omega=\edc\omega_{\cd}$. By orthogonality
as well as Lemma \ref{helmrefinedcorlem} and \eqref{epsuoonedf}, \eqref{epsuotwodf} we have
$$\norm{\omega}_{\Ltepsom{q}}^2
=\scp{\eps\,\omega}{\omega_{\cd}}_{\Ltom{q}}
\leq\Cedct{q}\norm{\eps\,\omega}_{\Ltom{q}}\norm{\edc\omega}_{\Ltom{q+1}}
\leq\Cedct{q}\epso\,\norm{\omega}_{\Ltepsom{q}}\norm{\edc\omega}_{\Ltom{q+1}},$$
and thus $\Cedcteps{q}\leq\Cedct{q}\epso$.
On the other hand, let $\omega\in\eDcom{q}\cap\cd\cDom{q+1}$.
According to Corollary \ref{helmrefinedcor} it holds
$$\eDcom{q}
=\eDczom{q}\oplus_{\Ltepsom{q}}\big(\eDcom{q}\cap\eps^{-1}\cd\cDom{q+1}\big)$$
and we can decompose
$$\omega=\omega_{0}+\omega_{\cd},\qquad
\omega_{0}\in\eDczom{q},\quad
\omega_{\cd}\in\eDcom{q}\cap\eps^{-1}\cd\cDom{q+1}$$
with $\edc\omega=\edc\omega_{\cd}$. By orthogonality
as well as Corollary \ref{helmrefinedcor} and \eqref{epsuoonedf}, \eqref{epsuotwodf} we have
$$\norm{\omega}_{\Ltom{q}}^2
=\scp{\omega}{\omega_{\cd}}_{\Ltom{q}}
\leq\epsu\,\norm{\omega}_{\Ltom{q}}\norm{\omega_{\cd}}_{\Ltepsom{q}}
\leq\Cedcteps{q}\epsu\,\norm{\omega}_{\Ltom{q}}\norm{\edc\omega}_{\Ltom{q+1}},$$
and thus $\Cedct{q}\leq\Cedcteps{q}\epsu$.
\end{proof}

It remains to estimate for all $q$ the constants $\Cedct{q}$.
For this we need the following result about regularity and Gaffney's inequality
in convex domains.

\begin{lem}
\label{gaffneyconvex}
Assume $\om$ additionally to be convex.
Let $\omega\in\eDcom{q}\cap\cDom{q}$ or $\omega\in\eDom{q}\cap\cDcom{q}$.
Then $\omega\in\Hoom{q}$ and 
$$\normltom{\na\vec\omega}^2
\leq\norm{\ed\omega}_{\Ltom{q+1}}^2
+\norm{\cd\omega}_{\Ltom{q-1}}^2.$$
\end{lem}

We will give a simple proof in Appendix \ref{appproofconvex}, 
only based on the well known corresponding result for smooth and convex domains,
see \eqref{friedrichsgaffneysmoothconvex}.
A proof of Lemma \ref{gaffneyconvex} can also be found in the nice paper of Mitrea
\cite[Theorem 5.5]{mitreamariusdirgafffriconvex}, see also
\cite[Corollary 5.6]{mitreamariusdirgafffriconvex}.
For $N=3$, partial and weaker results have been established earlier in
\cite[1.4 Satz, 5.5 Satz]{saranenmaxkegel},
\cite[Theorem 3.1]{saranenineqfried},
\cite[Corollary 3.6, Theorem 3.9]{giraultraviartbook},
\cite[Theorem 2.17]{amrouchebernardidaugegiraultvectorpot}.
Note that for all $\omega\in\Hocom{q}$ Gaffney's equation
\begin{align}
\label{gaffneyeq}
\norm{\na\vec\omega}_{\ltom}^2
=\norm{\ed\omega}_{\Ltom{q+1}}^2
+\norm{\cd\omega}_{\Ltom{q-1}}^2
\end{align}
holds, and that for convex domains all cohomology groups are trivial, i.e., $\Harmdepsom{q}=\{0\}$.

Now we can prove the key result for upper bounds.

\begin{lem}
\label{constestpoincareconvex}
Assume $\om$ additionally to be convex.
Then $\Cedct{q}\leq\cp$.
\end{lem}

\begin{proof}
By Lemma \ref{helmrefinedcorlem} we may pick 
$\zeta\in\cDom{q+1}\cap\edc\eDcom{q}=\cDom{q+1}\cap\eDczom{q+1}$.
Hence $\zeta=\edc\omega$ with some $\omega\in\eDcom{q}$.
Lemma \ref{gaffneyconvex} shows $\zeta\in\Hoom{q+1}$ and
for all $a\in\reals$ and all $I$ it holds 
$$\scp{\zeta_{I}}{a}_{\ltom}
=\scp{\zeta}{a\ed x^{I}}_{\Ltom{q+1}}
=a\,\scp{\edc\omega}{\ed x^{I}}_{\Ltom{q+1}}
=-a\,\scp{\omega}{\cd\ed x^{I}}_{\Ltom{q}}
=0.$$
Thus $\zeta_{I}\in\hoom\cap\reals^{\bot_{\ltom}}$ for all $I$
and we can apply the Poincar\'e estimate
and Lemma \ref{gaffneyconvex} to obtain
$$\norm{\zeta}_{\Ltom{q+1}}^2
=\sum_{I}\norm{\zeta_{I}}_{\ltom}^2
\leq\cp^2\sum_{I}\norm{\na\zeta_{I}}_{\ltom}^2
=\cp^2\norm{\na\vec\zeta}_{\ltom}^2
\leq\cp^2\norm{\cd\zeta}_{\Ltom{q}}^2.$$
Hence $\Cedct{q}\leq\cp$.
\end{proof}

A proof of Lemma \ref{constestpoincareconvex} can also be found in 
\cite[Corollary 5.10]{mitreamariusdirgafffriconvex},
where the estimates are equivalently formulated in terms of estimates for eigenvalues.
For $N=3$, the tangential boundary condition in $\mathring\hsymbol(\curl,\om)$,
and smooth convex domains the result has also been established in 
\cite[Theorem 3.1]{baozhouinvprobscat}.
In both papers, especially in \cite{baozhouinvprobscat},
the proof is more lengthy and complicated than our short proof.

For lower bounds we have the following.

\begin{lem}
\label{constestfriedrichs}
Assume $\om$ additionally to be topologically trivial.
Then $\Ct{q}\geq\cf$.
\end{lem}

\begin{proof}
As $\om$ is topologically trivial, all cohomology groups vanish.
Therefore, for all $u\in\hocom$ and some $I$ 
and with $\omega:=u\ed x^{I}\in\Hocom{q}\subset\eDcom{q}\cap\cDcom{q}$ 
we compute by \eqref{intromaxestelecdf} and \eqref{gaffneyeq}
\begin{align*}
\norm{u}_{\ltom}
=\norm{\omega}_{\Ltom{q}}
\leq\Ct{q}\big(\norm{\ed\omega}_{\Ltom{q+1}}^2
+\norm{\cd\omega}_{\Ltom{q-1}}^2\big)^{\oh}
=\Ct{q}\norm{\na\vec\omega}_{\ltom}
=\Ct{q}\norm{\na u}_{\ltom}.
\end{align*}
Thus $\cf\leq\Ct{q}$.
\end{proof}

\begin{lem}
\label{constestfriedrichstwo}
Assume $\om$ additionally to be topologically trivial.
Then $\displaystyle\Cteps{q}\geq\frac{\cf}{\epsh}$.
\end{lem}

\begin{proof}
It holds $\Ct{q}=\max\{\Cedct{q-1},\Cedct{q}\}$
and $\Cteps{q}=\max\{\tCedcteps{q-1},\Cedcteps{q}\}$.
If $\Ct{q}=\Cedct{q-1}$, then by Lemma \ref{constestremoveeps}
and Lemma \ref{constestfriedrichs}
$$\Cteps{q}
\geq\tCedcteps{q-1}
\geq\frac{\Cedct{q-1}}{\epso}
=\frac{\Ct{q}}{\epso}
\geq\frac{\cf}{\epsh}.$$
If $\Ct{q}=\Cedct{q}$, then by Lemma \ref{constestremoveeps}
and Lemma \ref{constestfriedrichs}
$$\Cteps{q}
\geq\Cedcteps{q}
\geq\frac{\Cedct{q}}{\epsu}
=\frac{\Ct{q}}{\epsu}
\geq\frac{\cf}{\epsh},$$
completing the proof.
\end{proof}

Combining Corollary \ref{constesthilbertdf}, Lemma \ref{constestremoveeps},
Lemma \ref{constestpoincareconvex}, Lemma \ref{constestfriedrichs}, and Lemma \ref{constestfriedrichstwo} 
we can formulate our main result.

\begin{theo}
\label{maintheo}
Assume $\om$ additionally to be convex.
Then for all $\omega\in\eDcom{q}\cap\eps^{-1}\cDom{q}$
\begin{align*}
\norm{\omega}_{\Ltepsom{q}}^2
&\leq\Cedcteps{q}^2\norm{\edc\omega}_{\Ltom{q+1}}^2
+\tCedcteps{q-1}^2\norm{\cd\eps\,\omega}_{\Ltom{q-1}}^2\\
&\leq\Cteps{q}^2\big(\norm{\edc\omega}_{\Ltom{q+1}}^2
+\norm{\cd\eps\,\omega}_{\Ltom{q-1}}^2\big).
\end{align*}
Moreover,
$$\frac{\Cedct{q-1}}{\epso}\leq\tCedcteps{q-1}\leq\Cedct{q-1}\epsu\leq\cp\epsu,\qquad
\frac{\Cedct{q}}{\epsu}\leq\Cedcteps{q}\leq\Cedct{q}\epso\leq\cp\epso$$
as well as
$$\frac{\cf}{\epsh}
\leq\Cteps{q}
=\max\{\tCedcteps{q-1},\Cedcteps{q}\}
\leq\cp\epsh,\qquad
\cp\leq\frac{\diam(\om)}{\pi}.$$
Especially, for $\eps=\id$ it holds for all $q$
\begin{align}
\label{cfcqmocqcp}
\tCedct{q}=\Cedct{q}\leq\cp,\qquad
\cf\leq\Ct{q}
=\max\{\Cedct{q-1},\Cedct{q}\}
\leq\cp\leq\frac{\diam(\om)}{\pi}.
\end{align}
\end{theo}

The corresponding theorem holds for the other boundary condition as well.

\begin{cor}
\label{maintheostar}
Assume $\om$ additionally to be convex.
Then for all $\omega\in\eDom{q}\cap\eps^{-1}\cDcom{q}$
\begin{align*}
\norm{\omega}_{\Ltepsom{q}}^2
&\leq\tCedctmu{N-q-1}^2\norm{\ed\omega}_{\Ltom{q+1}}^2
+\Cedctmu{N-q}^2\norm{\cdc\eps\,\omega}_{\Ltom{q-1}}^2\\
&\leq\Ctmu{N-q}^2\big(\norm{\ed\omega}_{\Ltom{q+1}}^2
+\norm{\cdc\eps\,\omega}_{\Ltom{q-1}}^2\big),
\end{align*}
where $\mu:=(-1)^{q(N-q)}*\eps^{-1}*$. Moreover,
$$\frac{\Cedct{N-q-1}}{\epsu}\leq\tCedctmu{N-q-1}\leq\Cedct{N-q-1}\epso\leq\cp\epso,\qquad
\frac{\Cedct{N-q}}{\epso}\leq\Cedctmu{N-q}\leq\Cedct{N-q}\epsu\leq\cp\epsu$$
as well as
$$\frac{\cf}{\epsh}
\leq\Ctmu{N-q}
=\max\{\tCedctmu{N-q-1},\Cedctmu{N-q}\}
\leq\cp\epsh,\qquad
\cp\leq\frac{\diam(\om)}{\pi}.$$
Especially, \eqref{cfcqmocqcp} holds for $\eps=\id$ and for all $q$.
\end{cor}

In the introduction we have denoted $\Ctmu{N-q}$ by $\Cneps{q}$.

\begin{proof}
Let $\omega\in\eDom{q}\cap\eps^{-1}\cDcom{q}$. Then $*\,\omega\in\cDom{N-q}$ and 
with $\mu^{-1}=(-1)^{q(N-q)}*\eps\,*$ we have
$$\zeta:=*\,\eps\,\omega=(-1)^{q(N-q)}*\eps**\,\omega\in\eDcom{N-q}\cap\mu^{-1}\cDom{N-q}.$$
As $\eps$ is admissible, so is $(-1)^{q(N-q)}*\eps\,*$ and hence also its inverse $\mu$. 
Theorem \ref{maintheostar} applied to $N-q$, $\zeta$, $\mu$ instead of $q$, $\omega$, $\eps$ shows
\begin{align*}
\norm{\zeta}_{\Ltmuom{N-q}}^2
&\leq\Cedctmu{N-q}^2\norm{\edc\zeta}_{\Ltom{N-q+1}}^2
+\tCedctmu{N-q-1}^2\norm{\cd\mu\,\zeta}_{\Ltom{N-q-1}}^2\\
&\leq\Ctmu{N-q}^2\big(\norm{\edc\zeta}_{\Ltom{N-q+1}}^2
+\norm{\cd\mu\,\zeta}_{\Ltom{N-q-1}}^2\big).
\end{align*}
Moreover, $*\,\eps\,*$ has the same properties \eqref{epsuoonedf}, \eqref{epsuotwodf} as $\eps$ 
and hence, as inverse, $\mu$ inherits these properties with 
$\epsu$ and $\epso$ interchanged. Note that, e.g.,
$$\scp{\mu\,\zeta}{\zeta}_{\Ltom{N-q}}
=\scp{\eps^{-1}*\zeta}{*\,\zeta}_{\Ltom{q}}
=\norm{\eps^{\moh}*\zeta}_{\Ltom{q}}^2
\leq\epsu^2\norm{\eps^{\moh}*\zeta}_{\Ltepsom{q}}^2
=\epsu^2\norm{\zeta}_{\Ltom{N-q}}^2$$
holds by \eqref{epsuotwodf}.
Hence the estimates for the constants follow immediately. Plugging in 
\begin{align*}
\norm{\zeta}_{\Ltmuom{N-q}}^2
&=\scp{\mu\,\zeta}{\zeta}_{\Ltom{N-q}}
=(-1)^{q(N-q)}\scp{*\,\eps^{-1}**\,\eps\,\omega}{*\,\eps\,\omega}_{\Ltom{N-q}}\\
&=\scp{\omega}{\eps\,\omega}_{\Ltom{q}}
=\norm{\omega}_{\Ltepsom{q}}^2,\\
\norm{\edc\zeta}_{\Ltom{N-q+1}}
&=\norm{\edc*\,\eps\,\omega}_{\Ltom{N-q+1}}
=\norm{\cdc\eps\,\omega}_{\Ltom{q-1}},\\
\norm{\cd\mu\,\zeta}_{\Ltom{N-q-1}}
&=\norm{\cd*\,\eps^{-1}**\,\eps\,\omega}_{\Ltom{N-q-1}}
=\norm{\ed\omega}_{\Ltom{q+1}}
\end{align*}
we obtain
\begin{align*}
\norm{\omega}_{\Ltepsom{q}}^2
&\leq\Cedctmu{N-q}^2\norm{\cdc\eps\,\omega}_{\Ltom{q-1}}^2
+\tCedctmu{N-q-1}^2\norm{\ed\omega}_{\Ltom{q+1}}^2\\
&\leq\Ctmu{N-q}^2\big(\norm{\cdc\eps\,\omega}_{\Ltom{q-1}}^2
+\norm{\ed\omega}_{\Ltom{q+1}}^2\big),
\end{align*}
completing the proof.
\end{proof}

The same transformation technique or 
just repeating the previous arguments
shows that Corollary \ref{helmrefinedcor}, especially the Friedrichs/Poincar\'e type estimates,
Corollary \ref{constesthilbertdf} and Lemma \ref{constestremoveeps}
hold for the other boundary condition placed on $\eps^{-1}\cDcom{q}$ as well.
More precisely, with $\mu$ as before and defining the (harmonic) Neumann forms by
$$\Harmnepsom{q}:=\eDzom{q}\cap\eps^{-1}\cDczom{q}$$
we have the following results.

\begin{cor}
\label{constesthilbertdfcdc}
For all $\omega\in\eDom{q}\cap\eps^{-1}\cDcom{q}\cap\Harmnepsom{q}^{\bot_{\Ltepsom{q}}}$
\begin{align*}
\norm{\omega}_{\Ltepsom{q}}^2
&\leq\tCedctmu{N-q-1}^2\norm{\ed\omega}_{\Ltom{q+1}}^2
+\Cedctmu{N-q}^2\norm{\cdc\eps\,\omega}_{\Ltom{q-1}}^2\\
&\leq\Ctmu{N-q}\big(\norm{\ed\omega}_{\Ltom{q+1}}^2+\norm{\cdc\eps\,\omega}_{\Ltom{q-1}}^2\big)^{\nicefrac{1}{2}}
\end{align*}
with $\Ctmu{N-q}=\max\{\tCedctmu{N-q-1},\Cedctmu{N-q}\}$. Especially,
\begin{align*}
\forall\,\xi&\in\eDom{q-1}\cap\cdc\cDcom{q}
&
\norm{\xi}_{\Ltom{q-1}}&\leq\Cedctmu{N-q}\norm{\ed\xi}_{\Ltepsom{q}},\\
\forall\,\omega&\in\eDom{q}\cap\eps^{-1}\cdc\cDcom{q+1}
&
\norm{\omega}_{\Ltepsom{q}}&\leq\tCedctmu{N-q-1}\norm{\ed\omega}_{\Ltom{q+1}},\\
\forall\,\omega&\in\eps^{-1}\cDcom{q}\cap\ed\eDom{q-1}
&
\norm{\omega}_{\Ltepsom{q}}&\leq\Cedctmu{N-q}\norm{\cdc\eps\,\omega}_{\Ltom{q-1}},\\
\forall\,\zeta&\in\cDcom{q+1}\cap\ed\eDom{q}
&
\norm{\zeta}_{\Ltom{q+1}}&\leq\tCedctmu{N-q-1}\norm{\eps^{-1}\cdc\zeta}_{\Ltepsom{q}}.
\end{align*}
\end{cor}

\begin{cor}
\label{constestremoveepscdc}
It holds
$$\frac{\Cedct{N-q-1}}{\epsu}\leq\tCedctmu{N-q-1}\leq\Cedct{N-q-1}\epso,\qquad
\frac{\Cedct{N-q}}{\epso}\leq\Cedctmu{N-q}\leq\Cedct{N-q}\epsu,$$
and
$$\frac{\min\{\Cedct{N-q-1},\Cedct{N-q}\}}{\epsh}
\leq\Ctmu{N-q}
=\max\{\tCedctmu{N-q-1},\Cedctmu{N-q}\}
\leq\max\{\Cedct{N-q-1},\Cedct{N-q}\}\epsh.$$
\end{cor}

\subsection{Some Remarks}

\begin{rem}
\mylabel{polyhedrarem}
Our results extend also to all possibly non-convex polyhedra which allow the 
$\Hoom{q}$-regularity in Lemma \ref{gaffneyconvex}
of the Maxwell spaces $\eDcom{q}\cap\cDom{q}$ and $\eDom{q}\cap\cDcom{q}$
or to domains whose boundaries consist of combinations of convex boundary parts 
and polygonal parts which allow the $\Hoom{q}$-regularity.
Such domains exist, depending on the special type of the singularities, 
which are not allowed to by too pointy, see, e.g., \cite{saranenmaxkegel,saranenmaxnichtglatt}.
It is well known that \eqref{gaffneyeq} even holds
for $\omega\in\Hoom{q}\cap\eDcom{q}$ or $\omega\in\Hoom{q}\cap\cDcom{q}$ 
if $\om$ is a polyhedron, since the unit normal is piecewise constant and hence the curvature is zero.
\end{rem}

\begin{rem}
\mylabel{convexandmorerem}
Let $\om$ be additionally convex and 
let us recall $\Cn{q}=\Ct{N-q}$ and \eqref{cfcqmocqcp}, especially
$$\cpc=\Ct{0}=\Cn{N}\leq\Ct{q},\Cn{q}\leq\Ct{N}=\Cn{0}=\cp\leq\frac{\diam(\om)}{\pi}.$$
\begin{itemize}
\item[\bf(i)] 
In generell, we conjecture $\cf<\Ct{q},\Cn{q}<\cp$ for $1\leq q\leq N-1$.
\item[\bf(ii)] 
As a byproduct, by
$$0<\mu_{2}=\frac{1}{\cp^2}\leq\frac{1}{\Ct{q}^2}\leq\frac{1}{\cf^2}=\lambda_{1}$$
we have shown a new proof of the well known fact, that
the first Dirichlet eigenvalue of the negative Laplacian $\lambda_{1}$
is not smaller than the second Neumann eigenvalue of the negative Laplacian $\mu_{2}$.
\end{itemize}
\end{rem}

\begin{rem}
\mylabel{nonconvexdomrem}
Our results extend to a certain class of non-convex domains,
so-called one-chart domains, as well.
For this, as before, let $\om\subset\rN$ be a bounded weak Lipschitz domain
and let $\Xi\subset\rN$ be a bounded and convex domain,
e.g., the unit cube or unit ball.
For example, $\om$ could be an L-shaped domain or a Fich\`era corner.
Moreover, we assume that there exists an orientation preserving bi-Lipschitz transformation
$\Phi:\Xi\to\om$ with inverse $\Psi:=\Phi^{-1}:\om\to\Xi$.

Then for $\omega\in\eDcom{q}\cap\eps^{-1}\cDom{q}$ we have
$$\Phi^{*}\omega\in\eDc{q}(\Xi)\cap\mu^{-1}\cD{q}(\Xi),\qquad
\mu:=(-1)^{qN-1}*\Phi^{*}*\eps\,\Psi^{*},$$
with
\begin{align}
\mylabel{PhisdLip}
\edc\Phi^{*}\omega
=\Phi^{*}\edc\omega,\qquad
\cd\mu\,\Phi^{*}\omega
=\pm*\ed\Phi^{*}*\eps\,\omega
=*\,\Phi^{*}*\cd\eps\,\omega,
\end{align}
see Appendix \ref{proofbiliptrans} for a proof of \eqref{PhisdLip} in the bi-Lipschitz case.
By the transformation formula, straight forward estimates, 
which we will carry out in Appendix \ref{appremnonconvex} as well, and Theorem \ref{maintheo} we get 
$$\norm{\omega}_{\Ltom{q}}
\leq\Cteps{q}
\big(\norm{\edc\omega}_{\Ltom{q+1}}^2
+\norm{\cd\eps\,\omega}_{\Ltom{q-1}}^2\big)^{\oh},$$
where
$$\Cteps{q}\leq c_{N}^3c_{\na\Phi,\na\Psi}^3\,\hat{\eps}\,c_{\mathsf{p},\Xi}$$
and $c_{\mathsf{p},\Xi}$ is the Poincar\'e constant for the convex domain $\Xi$,
$c_{N}$ depends just on $N$, and $c_{\na\Phi,\na\Phi}$ 
just on bounds for $\na\Phi$ and $\na\Psi$,
see \eqref{roughconstdef} in Appendix \ref{appremnonconvex} for more details.
These constants can be refined, if one takes a closer look 
at the actual dependence on $q$ and special algebraic operations on $\na\Phi$ and $\na\Psi$.
In Appendix \ref{appremnonconvexcl} we will present sharper estimates for the special case
$N=3$ and $q=1$ of vector proxy fields $\vec\omega$.

Using a partition of unity,
we can even extend our results to general bounded weak Lipschitz domains $\om\subset\rN$.
\end{rem}

\begin{acknow}
We cordially thank the anonymous referee for a very careful reading 
and valuable suggestions for improving the paper.
\end{acknow}

\bibliographystyle{plain} 
\bibliography{paule}

\appendix
\section{Proof of Lemma \ref{gaffneyconvex}}
\label{appproofconvex}

By the $*$-operator it is sufficient to discuss, e.g., $\omega\in\eDom{q}\cap\cDcom{q}$.
For a proof we follow the nice book of Grisvard, see \cite[Theorem 3.2.1.2, Theorem 3.2.1.3]{grisvardbook}.
This proof has been carried out in
\cite[Corollary 3.6, Theorem 3.9]{giraultraviartbook}
and \cite[Theorem 2.17]{amrouchebernardidaugegiraultvectorpot}
for the Maxwell case and $N=3$.
Our proof will avoid the misleading notion of traces and solutions of second order elliptic systems.
Let us note that in \cite[p. 834]{amrouchebernardidaugegiraultvectorpot}
the proof for $X_{N}(\om)$ is wrong. One cannot work in the space
$V_{T}(\om_{k})$ due to the solenoidal condition.
Working in the space $X_{T}(\om_{k})$ is needed, but this destroys their argument 
for the second order elliptic system for $\zeta$.
Our approach corrects these unconsistencies.

Let us pick a sequence of increasing, convex, and $\ci$-smooth subdomains $(\om_{n})\subset\om$
converging to $\om$, i.e.,
$$\om_{n}\subset\overline{\om}_{n}\subset\om_{n+1}\subset\dots\subset\om,\qquad
\dist(\om,\om_{n})=\dist(\p\om,\p\om_{n})\to0,$$
see, e.g., \cite[Lemma 3.2.1.1]{grisvardbook}. 
Of course, $\ct$-smooth is also sufficient.
For $\om_{n}$ we find
$\zeta_{n}\in\eD{q-1}(\om_{n})$ such that for all $\varphi\in\eD{q-1}(\om_{n})$
\begin{align}
\mylabel{pdedefomn}
\scp{\zeta_{n}}{\varphi}_{\eD{q-1}(\om_{n})}
=\scp{\cd\omega}{\varphi}_{\Lt{q-1}(\om_{n})}
+\scp{\omega}{\ed\varphi}_{\Lt{q}(\om_{n})},
\end{align}
which is a trivially well defined problem. 
Note $\scp{\zeta_{n}}{\varphi}_{\eD{q-1}(\om_{n})}
=\scp{\zeta_{n}}{\varphi}_{\Lt{q-1}(\om_{n})}
+\scp{\ed\zeta_{n}}{\ed\varphi}_{\Lt{q}(\om_{n})}$. Hence
\begin{align*}
\scp{\omega-\ed\zeta_{n}}{\ed\varphi}_{\Lt{q}(\om_{n})}
=\scp{\zeta_{n}-\cd\omega}{\varphi}_{\Lt{q-1}(\om_{n})}
\end{align*}
for all $\varphi\in\eD{q-1}(\om_{n})$, showing by \eqref{cDccharac} that
$\omega_{n}:=\omega-\ed\zeta_{n}\in\cDc{q}(\om_{n})$ and $\cd\omega_{n}=\cd\omega-\zeta_{n}$.
Moreover, $\omega_{n}\in\eD{q}(\om_{n})$ with $\ed\omega_{n}=\ed\omega$.
By \eqref{friedrichsgaffneysmoothconvex}
we have $\omega_{n}\in\genHo{q+1}(\om_{n})$ with
\begin{align}
\label{naomeganest}
\norm{\na\vec{\omega}_{n}}_{\lt(\om_{n})}^2
&\leq\norm{\ed\omega_{n}}_{\Lt{q+1}(\om_{n})}^2
+\norm{\cd\omega_{n}}_{\Lt{q-1}(\om_{n})}^2
=\norm{\ed\omega}_{\Lt{q+1}(\om_{n})}^2
+\norm{\cd\omega-\zeta_{n}}_{\Lt{q-1}(\om_{n})}^2.
\end{align}
By setting $\varphi=\zeta_{n}$ in \eqref{pdedefomn} we see
\begin{align}
\label{zetanest}
\begin{split}
\norm{\zeta_{n}}_{\eD{q-1}(\om_{n})}^2
&=\scp{\cd\omega}{\zeta_{n}}_{\Lt{q-1}(\om_{n})}
+\scp{\omega}{\ed\zeta_{n}}_{\Lt{q}(\om_{n})}\\
&\leq\norm{\cd\omega}_{\Lt{q-1}(\om_{n})}\norm{\zeta_{n}}_{\Lt{q-1}(\om_{n})}
+\norm{\omega}_{\Lt{q}(\om_{n})}\norm{\ed\zeta_{n}}_{\Lt{q}(\om_{n})}
\leq\norm{\omega}_{\cD{q}(\om_{n})}\norm{\zeta_{n}}_{\eD{q-1}(\om_{n})}
\end{split}
\end{align}
and thus
\begin{align}
\label{zetanesttwo}
\norm{\zeta_{n}}_{\eD{q-1}(\om_{n})}
&\leq\norm{\omega}_{\cD{q}(\om_{n})}
\leq\norm{\omega}_{\cDom{q}}.
\end{align}
Combining \eqref{naomeganest} and the equation part of \eqref{zetanest} we observe
\begin{align*}
\norm{\vec{\omega}_{n}}_{\ho(\om_{n})}^2
&=\norm{\omega_{n}}_{\Lt{q}(\om_{n})}^2
+\norm{\na\vec{\omega}_{n}}_{\lt(\om_{n})}^2
\leq\norm{\omega_{n}}_{\Lt{q}(\om_{n})}^2
+\norm{\ed\omega}_{\Lt{q+1}(\om_{n})}^2
+\norm{\cd\omega-\zeta_{n}}_{\Lt{q-1}(\om_{n})}^2\\
&=\norm{\omega}_{\Lt{q}(\om_{n})}^2
+\norm{\ed\zeta_{n}}_{\Lt{q}(\om_{n})}^2
+\norm{\ed\omega}_{\Lt{q+1}(\om_{n})}^2
+\norm{\cd\omega}_{\Lt{q-1}(\om_{n})}^2
+\norm{\zeta_{n}}_{\Lt{q-1}(\om_{n})}^2\\
&\qquad-2\scp{\omega}{\ed\zeta_{n}}_{\Lt{q}(\om_{n})}
-2\scp{\cd\omega}{\zeta_{n}}_{\Lt{q-1}(\om_{n})}\\
&=\norm{\omega}_{\eD{q}(\om_{n})\cap\cD{q}(\om_{n})}^2
+\norm{\zeta_{n}}_{\eD{q}(\om_{n})}^2
-2\norm{\zeta_{n}}_{\eD{q}(\om_{n})}^2
\leq\norm{\omega}_{\eD{q}(\om_{n})\cap\cD{q}(\om_{n})}^2
\end{align*}
and therefore
\begin{align}
\label{hoomeganest}
\begin{split}
\norm{\vec{\omega}_{n}}_{\ho(\om_{n})}
\leq\norm{\omega}_{\eD{q}(\om_{n})\cap\cD{q}(\om_{n})}
\leq\norm{\omega}_{\eDom{q}\cap\cDom{q}}.
\end{split}
\end{align}
Let us denote the extension by zero to $\om$ by $\tilde{\cdot}$.
Then by \eqref{zetanesttwo} and \eqref{hoomeganest}
the sequences $(\tilde{\zeta}_{n})$, $(\widetilde{\ed\zeta}_{n})$,
and $(\tilde{\vec{\omega}}_{n})$, $(\widetilde{\na\vec{\omega}}_{n})$
are bounded in $\Ltom{q-1}$, $\Ltom{q}$, resp. $\ltom$ and
we can extract weakly converging subsequences, again denoted by the index $n$, such that
\begin{align*}
\tilde{\zeta}_{n}
&\wto{\Ltom{q-1}}\zeta\in\Ltom{q-1},
&
\tilde{\vec{\omega}}_{n}
&\wto{\ltom}\vec{\hat{\omega}}\in\ltom,\\
(\widetilde{\ed\zeta}_{n})
&\wto{\Ltom{q}}\xi\in\Ltom{q},
&
\widetilde{\na\vec{\omega}}_{n}
&\wto{\ltom}\hat\Theta\in\ltom.
\end{align*}
Let $\psi\in\cicom$ and $n$ be large enough such that $\supp\psi\subset\om_{n}$.
Then $\psi\in\cic(\om_{n})$ and we calculate for $i=1,\dots,N$ and 
the $\ell$-th component $\vec{\hat{\omega}}_{\ell}$ of $\vec{\hat{\omega}}$
\begin{align*}
\scpltom{\vec{\hat{\omega}}_{\ell}}{\p_{i}\psi}
\ot\scpltom{\tilde{\vec{\omega}}_{n,\ell}}{\p_{i}\psi}
&=\scp{\vec{\omega}_{n,\ell}}{\p_{i}\psi}_{\lt(\om_{n})}\\
&=-\scp{\p_{i}\vec{\omega}_{n,\ell}}{\psi}_{\lt(\om_{n})}
=-\scpltom{\widetilde{\p_{i}\vec{\omega}}_{n,\ell}}{\psi}
\to-\scpltom{\hat{\Theta}_{i,\ell}}{\psi},
\end{align*}
yielding $\vec{\hat{\omega}}\in\hoom$ and $\na\vec{\hat{\omega}}=\hat{\Theta}$.
Analogously we obtain for $\phi\in\Cicom{q}$
with $\phi\in\Cic{q}(\om_{n})$ for $n$ large enough
\begin{align*}
\scp{\zeta}{\cd\phi}_{\Ltom{q-1}}
\ot\scp{\tilde{\zeta}_{n}}{\cd\phi}_{\Ltom{q-1}}
&=\scp{\zeta_{n}}{\cd\phi}_{\Lt{q-1}(\om_{n})}\\
&=-\scp{\ed\zeta_{n}}{\phi}_{\Lt{q}(\om_{n})}
=-\scp{\widetilde{\ed\zeta}_{n}}{\phi}_{\Ltom{q}}
\to-\scp{\xi}{\phi}_{\Ltom{q}},
\end{align*}
showing $\zeta\in\eDom{q-1}$ and $\ed\zeta=\xi$.
Moreover, for $\varphi\in\eDom{q-1}\subset\eD{q-1}(\om_{n})$
we have by \eqref{pdedefomn}
\begin{align*}
\scp{\zeta}{\varphi}_{\eDom{q-1}}
&=\scp{\zeta}{\varphi}_{\Ltom{q-1}}
+\scp{\ed\zeta}{\ed\varphi}_{\Ltom{q}}
\ot\scp{\tilde{\zeta}_{n}}{\varphi}_{\Ltom{q-1}}
+\scp{\widetilde{\ed\zeta}_{n}}{\ed\varphi}_{\Ltom{q}}
=\scp{\zeta_{n}}{\varphi}_{\eD{q-1}(\om_{n})}\\
&=\scp{\cd\omega}{\varphi}_{\Lt{q-1}(\om_{n})}
+\scp{\omega}{\ed\varphi}_{\Lt{q}(\om_{n})}
\to\scp{\cd\omega}{\varphi}_{\Ltom{q-1}}
+\scp{\omega}{\ed\varphi}_{\Ltom{q}}
=0,
\end{align*}
as $\omega\in\cDcom{q}$, where the last convergence follows by Lebesgue's dominated convergence theorem.
For $\varphi=\zeta$ we get $\norm{\zeta}_{\eDom{q-1}}=0$, i.e., $\zeta=0$.
Furthermore, we observe by \eqref{hoomeganest}
\begin{align*}
\norm{\vec{\hat\omega}}_{\hoom}^2
&=\scpltom{\vec{\hat\omega}}{\vec{\hat\omega}}
+\scpltom{\na\vec{\hat\omega}}{\na\vec{\hat\omega}}
\ot\scpltom{\vec{\hat\omega}}{\tilde{\vec{\omega}}_{n}}
+\scpltom{\na\vec{\hat\omega}}{\widetilde{\na\vec{\omega}}_{n}}\\
&=\scp{\vec{\hat\omega}}{\vec{\omega}_{n}}_{\lt(\om_{n})}
+\scp{\na\vec{\hat\omega}}{\na\vec{\omega}_{n}}_{\lt(\om_{n})}
\leq\norm{\vec{\hat\omega}}_{\ho(\om_{n})}
\norm{\vec{\omega}_{n}}_{\ho(\om_{n})}
\leq\norm{\vec{\hat\omega}}_{\hoom}
\norm{\omega}_{\eDom{q}\cap\cDom{q}},
\end{align*}
showing
\begin{align}
\mylabel{hoomnormomegahat}
\norm{\vec{\hat\omega}}_{\hoom}
\leq\norm{\omega}_{\eDom{q}\cap\cDom{q}}.
\end{align}
Finally, we have $\omega=\omega_{n}+\ed\zeta_{n}$ in $\om_{n}$, i.e., in $\om$
$$\chi_{\om_{n}}\omega
=\tilde{\omega}_{n}+\widetilde{\ed\zeta}_{n}
\wto{\Ltom{q}}\hat\omega+\ed\zeta
=\hat\omega.$$
On the other hand, by Lebesgue's dominated convergence theorem we see
$\chi_{\om_{n}}\omega\to\omega$ in $\Ltom{q}$. Thus
$\omega=\hat\omega\in\Hoom{q}$ and by \eqref{hoomnormomegahat}
\begin{align*}
\norm{\omega}_{\Hoom{q}}
&=\norm{\vec{\hat\omega}}_{\hoom}
\leq\norm{\omega}_{\eDom{q}\cap\cDom{q}},
\end{align*}
especially, 
\begin{align*}
\normltom{\na\vec\omega}^2
&\leq\norm{\ed\omega}_{\Ltom{q+1}}^2
+\norm{\cd\omega}_{\Ltom{q-1}}^2.
\end{align*}

\section{Calculations for Remark \ref{nonconvexdomrem}}
\label{appremnonconvex}

For a multi index $I$ of length $|I|=q$ (not necessarily ordered) it holds
\begin{align*}
\Phi^{*}\ed x^{I}
&=\Phi^{*}(\ed x^{i_{1}}\wedge\dots\wedge\ed x^{i_{q}})
=(\Phi^{*}\ed x^{i_{1}})\wedge\dots\wedge(\Phi^{*}\ed x^{i_{q}})
=(\ed\Phi_{i_{1}})\wedge\dots\wedge(\ed\Phi_{i_{q}})
=\ed\Phi^{I}\\
&=\sum_{j_{1},\dots,j_{q}}\p_{j_{1}}\Phi_{i_{1}}\dots\p_{j_{q}}\Phi_{i_{q}}
\ed x^{j_{1}}\wedge\dots\wedge\ed x^{j_{q}}
=\sum_{|J|=q}\p_{J}\Phi_{I}\ed x^{J}
\end{align*}
and especially
$$\Phi^{*}(\ed x^{1}\wedge\dots\wedge\ed x^{N})
=\det(\na\Phi)\ed x^{1}\wedge\dots\wedge\ed x^{N}.$$
For multi indices $I,J$ of length $q$ we have
\begin{align*}
(\Phi^{*}\ed x^{I})\wedge*(\Phi^{*}\ed x^{J})
&=\sum_{|K|=|L|=q}\p_{K}\Phi_{I}\p_{L}\Phi_{J}\ed x^{K}\wedge*\ed x^{L}\\
&=\sum_{|K|=q}(-1)^{\sigma_{K}}\p_{K}\Phi_{I}\p_{K}\Phi_{J}\ed x^{1}\wedge\dots\wedge\ed x^{N}.
\end{align*}
Hence for 
$$\omega=\sum_{I}\omega_{I}\ed x^{I},\quad
\Phi^{*}\omega=\sum_{I}\tilde\omega_{I}\,\Phi^{*}\ed x^{I},\quad
\tilde\omega:=\sum_{I}\tilde\omega_{I}\ed x^{I},\qquad
\tilde\omega_{I}:=\omega_{I}\circ\Phi$$
we compute
\begin{align*}
*\,|\omega|^2
=\omega\wedge*\,\bar{\omega}
&=\sum_{I,J}\omega_{I}\bar{\omega}_{J}\ed x^{I}\wedge*\ed x^{J}
=\sum_{I}\omega_{I}\bar{\omega}_{I}\ed x^{I}\wedge*\ed x^{I}
=|\vec\omega|^2\ed x^{1}\wedge\dots\wedge\ed x^{N},\\
*\,|\Phi^{*}\omega|^2
=\Phi^{*}\omega\wedge*\,\Phi^{*}\bar{\omega}
&=\sum_{I,J}\tilde\omega_{I}\bar{\tilde\omega}_{J}(\Phi^{*}\ed x^{I})\wedge*(\Phi^{*}\ed x^{J})\\
&=\sum_{I,J}\sum_{|K|=q}(-1)^{\sigma_{K}}\tilde\omega_{I}\bar{\tilde\omega}_{J}
\p_{K}\Phi_{I}\p_{K}\Phi_{J}\ed x^{1}\wedge\dots\wedge\ed x^{N},
\end{align*}
and thus
\begin{align*}
\norm{\vec\omega}_{\ltom}^2
&=\norm{\omega}_{\Ltom{q}}^2
=\int_{\om}*\,|\omega|^2
=\int_{\om}|\vec\omega|^2\ed x^{1}\wedge\dots\wedge\ed x^{N}
=\int_{\Xi}|\vec{\tilde\omega}|^2\Phi^{*}(\ed x^{1}\wedge\dots\wedge\ed x^{N})\\
&=\int_{\Xi}\det(\na\Phi)|\vec{\tilde\omega}|^2\ed x^{1}\wedge\dots\wedge\ed x^{N}
=\int_{\Xi}\det(\na\Phi)*|\tilde\omega|^2
=\int_{\Xi}\det(\na\Phi)|\vec{\tilde\omega}|^2,\\
\norm{\longvec{\Phi^{*}\omega}}_{\lt(\Xi)}^2
&=\norm{\Phi^{*}\omega}_{\Lt{q}(\Xi)}^2
=\int_{\Xi}*\,|\Phi^{*}\omega|^2
=\sum_{I,J}\sum_{|K|=q}(-1)^{\sigma_{K}}
\int_{\Xi}\tilde\omega_{I}\bar{\tilde\omega}_{J}
\p_{K}\Phi_{I}\p_{K}\Phi_{J}\ed x^{1}\wedge\dots\wedge\ed x^{N}\\
&=\sum_{I,J}\sum_{|K|=q}(-1)^{\sigma_{K}}
\int_{\Xi}\tilde\omega_{I}\bar{\tilde\omega}_{J}\p_{K}\Phi_{I}\p_{K}\Phi_{J}.
\end{align*}
Therefore, we get
\begin{align*}
\min_{\Xi}\det(\na\Phi)\,
\norm{\tilde\omega}_{\Lt{q}(\Xi)}^2
\leq\norm{\omega}_{\Ltom{q}}^2
&\leq\max_{\Xi}\det(\na\Phi)\,
\norm{\tilde\omega}_{\Lt{q}(\Xi)}^2,\\
\norm{\Phi^{*}\omega}_{\Lt{q}(\Xi)}^2
&\leq N^{q}\binom{N}{q}^{2}
\max_{\Xi}|\na\Phi|^{2q}\,
\norm{\tilde\omega}_{\Lt{q}(\Xi)}^2,
\end{align*}
where the second estimate is quite rough.
Combing both we see
\begin{align}
\label{normPhisomegaomega}
\norm{\Phi^{*}\omega}_{\Lt{q}(\Xi)}^2
&\leq c_{q,N,\na\Phi}\norm{\omega}_{\Ltom{q}}^2,
&
c_{q,N,\na\Phi}
&:=N^{q}\binom{N}{q}^{2}\frac{\max_{\Xi}|\na\Phi|^{2q}}{\min_{\Xi}\det(\na\Phi)},\\
\label{normPsiszetazeta}
\norm{\Psi^{*}\zeta}_{\Ltom{q}}^2
&\leq c_{q,N,\na\Psi}\norm{\zeta}_{\Lt{q}(\Xi)}^2,
&
c_{q,N,\na\Psi}
&:=N^{q}\binom{N}{q}^{2}\frac{\max_{\om}|\na\Psi|^{2q}}{\min_{\om}\det(\na\Psi)}
\end{align}
and with $\omega=\Psi^{*}\Phi^{*}\omega$
\begin{align*}
\norm{\omega}_{\Ltom{q}}^2
\leq c_{q,N,\na\Psi}
\norm{\Phi^{*}\omega}_{\Lt{q}(\Xi)}^2,\qquad
\norm{\zeta}_{\Lt{q}(\Xi)}^2
\leq c_{q,N,\na\Phi}\norm{\Psi^{*}\zeta}_{\Ltom{q}}^2.
\end{align*}
Now we calculate by Theorem \ref{maintheo}
\begin{align}
\label{calcmaxestomxidf}
\begin{split}
\norm{\omega}_{\Ltom{q}}^2
&\leq c_{q,N,\na\Psi}\norm{\Phi^{*}\omega}_{\Lt{q}(\Xi)}^2
\leq c_{q,N,\na\Psi}c_{\mathsf{p},\Xi}^2\,\hat{\mu}^2
\big(\norm{\edc\Phi^{*}\omega}_{\Lt{q+1}(\Xi)}^2
+\norm{\cd\mu\Phi^{*}\omega}_{\Lt{q-1}(\Xi)}^2\big)\\
&=c_{q,N,\na\Psi}c_{\mathsf{p},\Xi}^2\,\hat{\mu}^2
\big(\norm{\Phi^{*}\edc\omega}_{\Lt{q+1}(\Xi)}^2
+\norm{\Phi^{*}*\cd\eps\,\omega}_{\Lt{N-q+1}(\Xi)}^2\big)\\
&\leq c_{q,N,\na\Psi}c_{\mathsf{p},\Xi}^2\,\hat{\mu}^2
\big(c_{q+1,N,\na\Phi}\norm{\edc\omega}_{\Ltom{q+1}}^2
+c_{N-q+1,N,\na\Phi}\norm{\cd\eps\,\omega}_{\Ltom{q-1}}^2\big)\\
&\leq c_{q,N,\na\Psi}\max\{c_{q+1,N,\na\Phi},c_{N-q+1,N,\na\Phi}\}c_{\mathsf{p},\Xi}^2\,\hat{\mu}^2
\big(\norm{\edc\omega}_{\Ltom{q+1}}^2
+\norm{\cd\eps\,\omega}_{\Ltom{q-1}}^2\big)\\
&\leq c_{N}^4c_{\na\Phi,\na\Psi}^4\,\hat{\mu}^2c_{\mathsf{p},\Xi}^2
\big(\norm{\edc\omega}_{\Ltom{q+1}}^2
+\norm{\cd\eps\,\omega}_{\Ltom{q-1}}^2\big),
\end{split}
\end{align}
i.e.,
$$\Cteps{q}\leq c_{N}^2c_{\na\Phi,\na\Psi}^2\,\hat{\mu}\,c_{\mathsf{p},\Xi},$$
with very rough constants
\begin{align}
\mylabel{roughconstdef}
c_{N}:=N^{\nicefrac{N}{2}}N!,\qquad
c_{\na\Phi,\na\Psi}:=
\frac{\max\big[\max_{\Xi}|\na\Phi|,\max_{\om}|\na\Psi|,1\big]^{N}}
{\min\big[\min_{\Xi}\sqrt{\det(\na\Phi)},\min_{\om}\sqrt{\det(\na\Psi)},1\big]}.
\end{align}
So, it remains to estimate $\hat\mu$. For this we estimate for $\Phi^{*}\omega\in\Lt{q}(\Xi)$
\begin{align*}
\scp{\mu\,\Phi^{*}\omega}{\Phi^{*}\omega}_{\Lt{q}(\Xi)}
&=\pm\scp{*\,\Phi^{*}*\eps\,\omega}{\Phi^{*}\omega}_{\Lt{q}(\Xi)}
=\pm\scp{\Phi^{*}*\eps\,\omega}{*\,\Phi^{*}\omega}_{\Lt{N-q}(\Xi)}
=\pm\int_{\Xi}(\Phi^{*}*\eps\,\omega)\wedge(\Phi^{*}\bar\omega)\\
&=\pm\int_{\om}*\,\eps\,\omega\wedge\bar\omega
=\scp{\eps\,\omega}{\omega}_{\Ltom{q}}
\leq\epso^2\norm{\omega}_{\Ltom{q}}^2
\leq\epso^2c_{q,N,\na\Psi}\norm{\Phi^{*}\omega}_{\Lt{q}(\Xi)}^2,\\
\scp{\mu\,\Phi^{*}\omega}{\Phi^{*}\omega}_{\Lt{q}(\Xi)}
&=\scp{\eps\,\omega}{\omega}_{\Ltom{q}}
\geq\epsu^{-2}\norm{\omega}_{\Ltom{q}}^2
\geq\frac{1}{\epsu^2c_{q,N,\na\Phi}}\norm{\Phi^{*}\omega}_{\Lt{q}(\Xi)}^2,
\end{align*}
and observe
$$\hat\mu
\leq\max\{\epso\sqrt{c_{q,N,\na\Psi}},\epsu\sqrt{c_{q,N,\na\Phi}}\}
\leq\epsh\max\{\sqrt{c_{q,N,\na\Psi}},\sqrt{c_{q,N,\na\Phi}}\}
\leq\epsh\,c_{N}\,c_{\na\Phi,\na\Psi}.$$
Finally, this shows
$$\Cteps{q}\leq c_{N}^3c_{\na\Phi,\na\Psi}^3\,\hat{\eps}\,c_{\mathsf{p},\Xi}.$$

\subsection{Classical Vector Analysis}
\label{appremnonconvexcl}

Some of the latter estimates are very rough. Let us take a closer look 
at the classical case of vector analysis, i.e., at the special case of $N=3$ and $q=1$.
By \eqref{PhisdLip}, see also Appendix \ref{proofbiliptrans} for more details and a rigorous proof,
we know that $\omega$ in $\eDom{q}$ resp. $\eDcom{q}$ implies
$\Phi^{*}\omega$ in $\eD{q}(\Xi)$ resp. $\eDc{q}(\Xi)$ 
with $\ed\Phi^{*}\omega=\Phi^{*}\ed\omega$. 
For $N=3$ and $q=1$ this means for the vector proxy field 
$\vec\omega\in\mathring\hsymbol(\curl,\om)\cong\eDcom{1}$
that 
$$\longvec{\Phi^{*}\omega}=\na\Phi\,\vec{\tilde\omega}\in\mathring\hsymbol(\curl,\Xi)\cong\eDc{1}(\Xi)$$ 
with 
\begin{align}
\mylabel{veccurlformula}
\curl(\na\Phi\,\vec{\tilde\omega})
=\longvec{\ed\Phi^{*}\omega}
=\longvec{\Phi^{*}\ed\omega}
=\adjt(\na\Phi)\widetilde{\curl\vec\omega},
\end{align}
where $\adj(A)$ denotes the adjunct matrix of $A\in\rttt$.
If $A$ is invertible it holds $\adj(A)=(\det A)A^{-1}$.
For $q=N-1=2$ we have for the vector proxy field 
$\vec\omega\in\hsymbol(\div,\om)\cong\eDom{2}$
that 
$$\longvec{\Phi^{*}\omega}
=\adjt(\na\Phi)\,\vec{\tilde\omega}\in\hsymbol(\div,\Xi)\cong\eD{2}(\Xi)$$ 
with 
$$\div\big(\adjt(\na\Phi)\,\vec{\tilde\omega}\big)
=\longvec{\ed\Phi^{*}\omega}
=\longvec{\Phi^{*}\ed\omega}
=\det(\na\Phi)\widetilde{\div\vec\omega}.$$
Thus for $\vec\omega\in\mathring\hsymbol(\curl,\om)\cap\eps^{-1}\hsymbol(\div,\om)$
we have 
$$\na\Phi\,\vec{\tilde\omega}\in\mathring\hsymbol(\curl,\Xi)\cap\mu^{-1}\hsymbol(\div,\Xi),\qquad
\mu:=\frac{1}{\det(\na\Phi)}\adjt(\na\Phi)\,\tilde\eps\,\adj(\na\Phi),$$
with \eqref{veccurlformula} and 
\begin{align*}
\div(\mu\na\Phi\,\vec{\tilde\omega})
&=\div\big(\adjt(\na\Phi)\,\tilde\eps\,\vec{\tilde\omega}\big)
=\det(\na\Phi)\widetilde{\div\eps\,\vec\omega}.
\end{align*}
Now we can compute \eqref{calcmaxestomxidf} more carefully by
\begin{align}
\label{calcmaxestomxi}
\begin{split}
\norm{\vec\omega}_{\ltom}^2
&=\int_{\om}\norm{\vec\omega}^2
=\int_{\Xi}\det(\na\Phi)\norm{\vec{\tilde\omega}}^2
\leq\int_{\Xi}\det(\na\Phi)\bnorm{(\na\Phi)^{-1}}^2\norm{\na\Phi\,\vec{\tilde\omega}}^2\\
&=\int_{\Xi}\frac{1}{\det(\na\Phi)}\bnorm{\adj(\na\Phi)}^2\norm{\na\Phi\,\vec{\tilde\omega}}^2
\leq\hat{c}_{\na\Phi}^2\norm{\na\Phi\,\vec{\tilde\omega}}_{\lt(\Xi)}^2\\
&\leq\hat{c}_{\na\Phi}^2c_{\mathsf{m,t},\mu,\Xi}^2
\big(\bnorm{\curl(\na\Phi\,\vec{\tilde\omega})}_{\lt(\Xi)}^2
+\bnorm{\div(\mu\na\Phi\,\vec{\tilde\omega})}_{\lt(\Xi)}^2\big)\\
&=\hat{c}_{\na\Phi}^2c_{\mathsf{m,t},\mu,\Xi}^2
\big(\bnorm{\adjt(\na\Phi)\widetilde{\curl\vec\omega}}_{\lt(\Xi)}^2
+\bnorm{\det(\na\Phi)\widetilde{\div\eps\,\vec\omega}}_{\lt(\Xi)}^2\big)\\
&=\hat{c}_{\na\Phi}^2c_{\mathsf{m,t},\mu,\Xi}^2
\big(\int_{\Xi}\bnorm{\adjt(\na\Phi)\widetilde{\curl\vec\omega}}^2
+\int_{\Xi}\bnorm{\det(\na\Phi)\widetilde{\div\eps\,\vec\omega}}^2\big)\\
&\leq\hat{c}_{\na\Phi}^2c_{\mathsf{m,t},\mu,\Xi}^2
\big(\hat{c}_{\na\Phi}^2\int_{\Xi}\det(\na\Phi)\norm{\widetilde{\curl\vec\omega}}^2
+c_{\det(\na\Phi)}^2\int_{\Xi}\det(\na\Phi)\norm{\widetilde{\div\eps\,\vec\omega}}^2\big)\\
&=\hat{c}_{\na\Phi}^2c_{\mathsf{m,t},\mu,\Xi}^2
\big(\hat{c}_{\na\Phi}^2\norm{\curl\vec\omega}_{\ltom}^2
+c_{\det(\na\Phi)}^2\norm{\div\eps\,\vec\omega}_{\ltom}^2\big),
\end{split}
\end{align}
where
\begin{align*}
c_{\det(\na\Phi)}
&:=\max_{\Xi}\sqrt{\det(\na\Phi)},\\
\hat{c}_{\na\Phi}
&:=\max_{\Xi}\frac{\big|\adj(\na\Phi)\big|}{\sqrt{\det(\na\Phi)}}
=\max_{\Xi}\sqrt{\det(\na\Phi)}\big|(\na\Phi)^{-1}\big|
\leq c_{\det(\na\Phi)}\max_{\Xi}\big|(\na\Phi)^{-1}\big|.
\end{align*}
Therefore, we have
$$\cmteps
\leq\hat{c}_{\na\Phi}
\max\{\hat{c}_{\na\Phi},c_{\det(\na\Phi)}\}
c_{\mathsf{m,t},\mu,\Xi},\qquad
c_{\mathsf{m,t},\mu,\Xi}
\leq\hat\mu\,c_{\mathsf{p},\Xi},$$
and it remains to estimate $\hat\mu$.
For this we compute for $\vec{\tilde\omega}\in\lt(\Xi)$
\begin{align*}
\scp{\mu\,\vec{\tilde\omega}}{\vec{\tilde\omega}}_{\lt(\Xi)}
&=\int_{\Xi}\mu\,\vec{\tilde\omega}\cdot\vec{\bar{\tilde\omega}}
=\int_{\Xi}\det(\na\Phi)\big((\na\Phi)^{-\top}\tilde\eps\,(\na\Phi)^{-1}\vec{\tilde\omega}\big)\cdot\vec{\bar{\tilde\omega}}\\
&=\int_{\Xi}\det(\na\Phi)\big(\tilde\eps\,(\na\Phi)^{-1}\vec{\tilde\omega}\big)\cdot(\na\Phi)^{-1}\vec{\bar{\tilde\omega}}
=\int_{\om}(\eps\na\Psi\,\vec\omega)\cdot\na\Psi\,\vec{\bar{\omega}}
=\scp{\eps\na\Psi\,\vec\omega}{\na\Psi\,\vec\omega}_{\ltom}
\end{align*}
and estimate
\begin{align*}
\scp{\mu\,\vec{\tilde\omega}}{\vec{\tilde\omega}}_{\lt(\Xi)}
&\leq\epso^2\norm{\na\Psi\,\vec\omega}_{\ltom}^2
=\epso^2\int_{\om}\norm{\na\Psi\,\vec\omega}^2
=\epso^2\int_{\Xi}\det(\na\Phi)\norm{(\na\Phi)^{-1}\vec{\tilde\omega}}^2\\
&\leq\epso^2\int_{\Xi}\det(\na\Phi)\norm{(\na\Phi)^{-1}}^2\norm{\vec{\tilde\omega}}^2
\leq\epso^2\hat{c}_{\na\Phi}^2\int_{\Xi}\norm{\vec{\tilde\omega}}^2
=\epso^2\hat{c}_{\na\Phi}^2\norm{\vec{\tilde\omega}}_{\lt(\Xi)}^2,\\
\scp{\mu\,\vec{\tilde\omega}}{\vec{\tilde\omega}}_{\lt(\Xi)}
&\geq\epsu^{-2}\norm{\na\Psi\,\vec\omega}_{\ltom}^2
=\epsu^{-2}\int_{\Xi}\det(\na\Phi)\norm{(\na\Phi)^{-1}\vec{\tilde\omega}}^2\\
&\geq\epsu^{-2}\int_{\Xi}\frac{\det(\na\Phi)}{\norm{\na\Phi}^{2}}\norm{\vec{\tilde\omega}}^2
\geq\epsu^{-2}\check{c}_{\na\Phi}^{-2}\int_{\Xi}\norm{\vec{\tilde\omega}}^2
=\frac{1}{\epsu^2\check{c}_{\na\Phi}^2}\norm{\vec{\tilde\omega}}_{\lt(\Xi)}^2,
\end{align*}
where
\begin{align*}
\check{c}_{\na\Phi}
&:=\max_{\Xi}\frac{|\na\Phi|}{\sqrt{\det(\na\Phi)}}
=\frac{1}{\min_{\Xi}\frac{\sqrt{\det(\na\Phi)}}{|\na\Phi|}}.
\end{align*}
Finally, we obtain
$$\hat\mu
\leq\max\{\epso\,\hat{c}_{\na\Phi},\epsu\,\check{c}_{\na\Phi}\}
\leq\epsh\max\{\hat{c}_{\na\Phi},\check{c}_{\na\Phi}\}$$
and hence
\begin{align}
\mylabel{resultPhigen}
\cmteps
\leq\hat{c}_{\na\Phi}
\max\{\hat{c}_{\na\Phi},c_{\det(\na\Phi)}\}
\max\{\hat{c}_{\na\Phi},\check{c}_{\na\Phi}\}
\,\epsh\,c_{\mathsf{p},\Xi}
\leq\max\{\hat{c}_{\na\Phi},\check{c}_{\na\Phi},c_{\det(\na\Phi)}\}^3
\,\epsh\,c_{\mathsf{p},\Xi}.
\end{align}

Especially for $\Phi(x):=r\,x$ with $r>0$ we have
\begin{align*}
\norm{\vec\omega}_{\ltom}^2
&=\int_{\om}\norm{\vec\omega}^2
=\int_{\Xi}\det(\na\Phi)\norm{\vec{\tilde\omega}}^2
=r\norm{\na\Phi\,\vec{\tilde\omega}}_{\lt(\Xi)}^2\\
&\leq rc_{\mathsf{m,t},\mu,\Xi}^2
\big(\bnorm{\curl(\na\Phi\,\vec{\tilde\omega})}_{\lt(\Xi)}^2
+\bnorm{\div(\mu\na\Phi\,\vec{\tilde\omega})}_{\lt(\Xi)}^2\big)\\
&=rc_{\mathsf{m,t},\mu,\Xi}^2
\big(\int_{\Xi}\bnorm{\adjt(\na\Phi)\widetilde{\curl\vec\omega}}^2
+\int_{\Xi}\bnorm{\det(\na\Phi)\widetilde{\div\eps\,\vec\omega}}^2\big)\\
&=rc_{\mathsf{m,t},\mu,\Xi}^2
\big( r\int_{\Xi}\det(\na\Phi)\norm{\widetilde{\curl\vec\omega}}^2
+r^{3}\int_{\Xi}\det(\na\Phi)\norm{\widetilde{\div\eps\,\vec\omega}}^2\big)\\
&=r^{2}c_{\mathsf{m,t},\mu,\Xi}^2
\big(\norm{\curl\vec\omega}_{\ltom}^2
+r^2\norm{\div\eps\,\vec\omega}_{\ltom}^2\big)
\end{align*}
and
\begin{align*}
\scp{\mu\,\vec{\tilde\omega}}{\vec{\tilde\omega}}_{\lt(\Xi)}
&=\int_{\Xi}\mu\,\vec{\tilde\omega}\cdot\vec{\bar{\tilde\omega}}
=\int_{\Xi}\det(\na\Phi)\big((\na\Phi)^{-\top}\tilde\eps\,(\na\Phi)^{-1}\vec{\tilde\omega}\big)\cdot\vec{\bar{\tilde\omega}}
=r^{-2}\int_{\Xi}\det(\na\Phi)(\tilde\eps\,\vec{\tilde\omega})\cdot\vec{\bar{\tilde\omega}}\\
&=r^{-2}\int_{\om}(\eps\,\vec\omega)\cdot\,\vec{\bar{\omega}}
=r^{-2}\scp{\eps\,\vec\omega}{\,\vec\omega}_{\ltom},\\
\scp{\mu\,\vec{\tilde\omega}}{\vec{\tilde\omega}}_{\lt(\Xi)}
&\leq r^{-2}\epso^2\norm{\vec\omega}_{\ltom}^2
=r^{-2}\epso^2\int_{\Xi}\det(\na\Phi)\norm{\vec{\tilde\omega}}^2
=r\epso^2\norm{\vec{\tilde\omega}}_{\lt(\Xi)}^2,\\
\scp{\mu\,\vec{\tilde\omega}}{\vec{\tilde\omega}}_{\lt(\Xi)}
&\geq r\epsu^{-2}\norm{\vec{\tilde\omega}}_{\lt(\Xi)}^2,
\end{align*}
i.e., 
$\displaystyle\hat\mu
\leq\max\{\sqrt{r}\epso,\epsu/\sqrt{r}\}
\leq\frac{\max\{r,1\}}{\sqrt{r}}\epsh$, 
which shows
$$\cmteps
\leq r\max\{1,r\}
c_{\mathsf{m,t},\mu,\Xi}
\leq r\max\{1,r\}
\,\hat\mu\,c_{\mathsf{p},\Xi}
\leq\sqrt{r}\max\{1,r\}^2
\,\epsh\,c_{\mathsf{p},\Xi}.$$
On the other hand, \eqref{resultPhigen} gives with 
$c_{\det(\na\Phi)}=r^{\nicefrac{3}{2}}$, 
$\hat{c}_{\na\Phi}=\sqrt{3}r^{\nicefrac{1}{2}}$,
$\check{c}_{\na\Phi}=\sqrt{3}r^{-\nicefrac{1}{2}}$ 
the less sharp estimate
$$\cmteps
\leq3\sqrt{3}r^{\nicefrac{3}{2}}\max\{1,r^2\}^3
\,\epsh\,c_{\mathsf{p},\Xi}.$$

\section{Proof of \eqref{PhisdLip} in the Bi-Lipschitz Case.}
\label{proofbiliptrans}

\subsection{Without Boundary Conditions}
\label{appwithoutbc}

For this, let $\omega=\sum_{I}\omega_{I}\ed x^{I}\in\eDom{q}$. We have to prove
$\Phi^{*}\omega\in\eD{q}(\Xi)$ with $\ed\Phi^{*}\omega=\Phi^{*}\ed\omega$.
Let us first assume $\omega\in\Cic{q}(\rN)$, i.e., $\omega_{I}\in\cic(\rN)$
for all $I$. 
By Appendix \ref{appremnonconvex} we have
\begin{align*}
\ed\Phi_{j}
&=\sum_{i}\p_{i}\Phi_{j}\ed x^{i},
&
\Phi^{*}\omega
&=\sum_{I}\tilde{\omega}_{I}\Phi^{*}\ed x^{I}
=\sum_{I}\tilde{\omega}_{I}(\ed\Phi_{i_{1}})\wedge\dots\wedge(\ed\Phi_{i_{q}}),\\
&&
\ed\omega
&=\sum_{I,j}\p_{j}\omega_{I}(\ed x_{j})\wedge(\ed x^{I}).
\end{align*}
By Rademacher's theorem we know that $\tilde\omega_{I}=\omega_{I}\circ\Phi$
and $\Phi_{j}$ belong to $\mathsf{C}^{0,1}(\Xi)\subset\ho(\Xi)$ and that the chain rule holds, i.e.,
$\p_{i}\tilde\omega_{I}=\sum_{j}\widetilde{\p_{j}\omega_{I}}\p_{i}\Phi_{j}$.
As $\Phi_{j}\in\ho(\Xi)$ we get $\ed\Phi_{j}\in\eDz{1}(\Xi)$ by
$$\scp{\ed\Phi_{j}}{\cd\varphi}_{\Lt{1}(\Xi)}
=-\scp{\Phi_{j}}{\cd\cd\varphi}_{\Lt{0}(\Xi)}
=0$$
for all $\varphi\in\Cic{2}(\Xi)$. 
Thus by definition we see
\begin{align*}
\ed\Phi^{*}\omega
&=\sum_{I}(\ed\tilde{\omega}_{I})\wedge(\ed\Phi_{i_{1}})\wedge\dots\wedge(\ed\Phi_{i_{q}})
=\sum_{I,i}\p_{i}\tilde{\omega}_{I}
(\ed x^{i})\wedge(\ed\Phi_{i_{1}})\wedge\dots\wedge(\ed\Phi_{i_{q}})\\
&=\sum_{I,i,j}\widetilde{\p_{j}\omega_{I}}\p_{i}\Phi_{j}
(\ed x^{i})\wedge(\ed\Phi_{i_{1}})\wedge\dots\wedge(\ed\Phi_{i_{q}})
=\sum_{I,j}\widetilde{\p_{j}\omega_{I}}
(\ed\Phi_{j})\wedge(\ed\Phi_{i_{1}})\wedge\dots\wedge(\ed\Phi_{i_{q}}).
\end{align*}
On the other hand it holds
\begin{align*}
\Phi^{*}\ed\omega
&=\sum_{I,j}\widetilde{\p_{j}\omega_{I}}(\Phi^{*}\ed x_{j})\wedge(\Phi^{*}\ed x^{I})
=\sum_{I,j}\widetilde{\p_{j}\omega_{I}}(\ed\Phi_{j})\wedge(\ed\Phi_{i_{1}})\wedge\dots\wedge(\ed\Phi_{i_{q}}).
\end{align*}
Therefore, $\Phi^{*}\omega\in\eD{q}(\Xi)$ with $\ed\Phi^{*}\omega=\Phi^{*}\ed\omega$.
For general $\omega\in\eDom{q}$ we pick $\phi\in\Cic{q+1}(\Xi)$.
The first part of the proof (for $\omega=*\,\phi$ and $\Phi=\Psi$) shows
$\Psi^{*}*\,\phi\in\eDom{N-q-1}$ with $\ed\Psi^{*}*\,\phi=\Psi^{*}\ed*\,\phi$.
As $\supp*\,\Psi^{*}*\,\phi$ is a compact subset of $\om$,
standard mollification yields a sequence $(\Phi_{n})\subset\Cic{q+1}(\om)$
with $\Phi_{n}\to*\,\Psi^{*}*\,\phi$ in $\cDom{q+1}$. Then
\begin{align*}
\scp{\Phi^{*}\omega}{\cd\phi}_{\Lt{q}(\Xi)}
&=\int_{\Xi}\Phi^{*}\omega\wedge*\cd\phi
=\pm\int_{\Xi}\Phi^{*}\omega\wedge\Phi^{*}\Psi^{*}\ed*\,\phi
=\pm\int_{\Xi}\Phi^{*}(\omega\wedge\Psi^{*}\ed*\,\phi)\\
&=\pm\int_{\om}\omega\wedge\Psi^{*}\ed*\,\phi
=\pm\int_{\om}\omega\wedge\ed\Psi^{*}*\,\phi
=\pm\scp{\omega}{\cd*\Psi^{*}*\,\phi}_{\Lt{q}(\om)}\\
&\uparrow\;\pm\scp{\omega}{\cd\Phi_{n}}_{\Lt{q}(\om)}
=\pm\scp{\ed\omega}{\Phi_{n}}_{\Lt{q+1}(\om)}\\
&\downarrow\;\pm\scp{\ed\omega}{*\Psi^{*}*\phi}_{\Lt{q+1}(\om)}
=\pm\int_{\om}\ed\omega\wedge\Psi^{*}*\,\phi\\
&=\pm\int_{\Xi}\Phi^{*}(\ed\omega\wedge\Psi^{*}*\,\phi)
=\pm\int_{\Xi}(\Phi^{*}\ed\omega)\wedge*\,\phi
=-\scp{\Phi^{*}\ed\omega}{\phi}_{\Lt{q+1}(\Xi)}
\end{align*}
and hence $\Phi^{*}\omega\in\eD{q}(\Xi)$ with $\ed\Phi^{*}\omega=\Phi^{*}\ed\omega$.
Finally, for $\omega\in\eps^{-1}\cDom{q}$
we have $\eps\,\omega\in\cDom{q}$ and $*\,\eps\,\omega\in\eDom{N-q}$.
Therefore, $\Phi^{*}*\,\eps\,\omega\in\eD{N-q}(\Xi)$ and
$\ed\Phi^{*}*\,\eps\,\omega=\Phi^{*}\ed*\,\eps\,\omega=\pm\Phi^{*}*\cd\,\eps\,\omega$ 
by the latter considerations. Hence
$$*\,\Phi^{*}*\cd\,\eps\,\omega
=\pm*\,\ed\Phi^{*}*\,\eps\,\omega
=\pm\cd(\underbrace{*\,\Phi^{*}*\,\eps\,\Psi^{*}}_{=\pm\mu})\,\Phi^{*}\omega$$
and $\mu\,\Phi^{*}\omega\in\cD{q}(\Xi)$. By \eqref{normPhisomegaomega} we see
$$\norm{\Phi^{*}\omega}_{\Lt{q}(\Xi)}
\leq c\,\norm{\omega}_{\Ltom{q}},\qquad
\norm{\ed\Phi^{*}\omega}_{\Lt{q+1}(\Xi)}
=\norm{\Phi^{*}\ed\omega}_{\Lt{q+1}(\Xi)}
\leq c\,\norm{\ed\omega}_{\Ltom{q+1}}$$
and
$$\norm{\cd\mu\,\Phi^{*}\omega}_{\Lt{q-1}(\Xi)}
=\norm{\ed\Phi^{*}*\eps\,\omega}_{\Lt{N-q+1}(\Xi)}
\leq c\,\norm{\ed*\eps\,\omega}_{\Ltom{N-q+1}}
=c\,\norm{\cd\eps\,\omega}_{\Ltom{q-1}}.$$

\subsection{With Boundary Conditions}

Let $\omega\in\eDcom{q}$ and $(\omega_{n})\subset\Cicom{q}$ 
with $\omega_{n}\to\omega$ in $\eDom{q}$. By Appendix \ref{appwithoutbc}
we know $\Phi^{*}\omega,\Phi^{*}\omega_{n}\in\eD{q}(\Xi)$ 
with $\ed\Phi^{*}\omega_{n}=\Phi^{*}\ed\omega_{n}$ as well as $\ed\Phi^{*}\omega=\Phi^{*}\ed\omega$.
Since $\Phi^{*}\omega_{n}=\sum_{I}\tilde{\omega}_{n,I}\Phi^{*}\ed x^{I}$ holds,
$\Phi^{*}\omega_{n}$ has compact support in $\Xi$. By standard mollification
we see $\Phi^{*}\omega_{n}\in\eDc{q}(\Xi)$.
Moreover, $\Phi^{*}\omega_{n}\to\Phi^{*}\omega$ in $\eD{q}(\Xi)$
as $\Phi^{*}\omega_{n}\to\Phi^{*}\omega$ in $\Lt{q}(\Xi)$ and 
$$\ed\Phi^{*}\omega_{n}=\Phi^{*}\ed\omega_{n}\to\Phi^{*}\ed\omega=\ed\Phi^{*}\omega$$
in $\Lt{q+1}(\Xi)$ by \eqref{normPhisomegaomega}. Therefore
$\Phi^{*}\omega\in\eDc{q}(\Xi)$ with $\edc\Phi^{*}\omega=\Phi^{*}\edc\omega$.

\end{document}